\documentclass[11pt]{amsart}
\usepackage[utf8]{inputenc}
\usepackage[T1]{fontenc}
\usepackage[a4paper, margin = 2.7cm]{geometry}
\usepackage{baskervillef}
\usepackage{cite}
\usepackage{amsmath}
\usepackage{mathrsfs}
\usepackage[english]{babel}
\usepackage{amssymb}
\usepackage{amsthm}
\usepackage{bm}
\usepackage{tikz}
\usetikzlibrary{positioning,decorations.markings,backgrounds,fit,shapes.geometric,graphs}
\usepackage{tikz-cd}
\usepackage{tikzsymbols}
\usepackage{mdframed}
\usepackage{framed}
\usepackage{mathtools}
\usepackage{halloweenmath}
\usepackage{fancyhdr}
\usepackage{enumerate}
\usepackage{graphicx}
\usepackage{hyperref}
\usepackage{cleveref}
\usepackage{colortbl}
\usepackage{epigraph}
\usepackage{subfig}
\usepackage{gensymb}
\usepackage{comment}
\usepackage[font=small, labelfont=bf]{caption}
\usepackage{chngcntr}
\usepackage{xcolor}
\usepackage{graphicx}
\usepackage{framed}
\usepackage[normalem]{ulem}
\usepackage{mathdots}
\usepackage[symbol]{footmisc}
\usepackage{booktabs} 

\theoremstyle{plain}
\newtheorem{proposition}{Proposition}[section]
\newtheorem{corollary}[proposition]{Corollary}
\newtheorem{lemma}[proposition]{Lemma}
\newtheorem{theorem}[proposition]{Theorem}

\theoremstyle{definition}
\newtheorem{definition}[proposition]{Definition}
\newtheorem{example}[proposition]{Example}
\newtheorem*{example*}{Example}

\theoremstyle{remark}

\numberwithin{equation}{section}
\numberwithin{figure}{section}

\newcommand*\openquote{\makebox(25,-22){\scalebox{5}{\color{white}``}}}
\newcommand*\closequote{\makebox(25,-22){\scalebox{5}{\color{white}''}}}
\colorlet{shadecolor}{gray!20}

\makeatletter
\newif\if@right
\def\shadequote{\@righttrue\shadequote@i}
\def\shadequote@i{\begin{snugshade}\begin{quote}\openquote}
\def\endshadequote{%
  \if@right\hfill\fi\closequote\end{quote}\end{snugshade}}
\@namedef{shadequote*}{\@rightfalse\shadequote@i}
\@namedef{endshadequote*}{\endshadequote}
\makeatother

\DeclareMathOperator{\Hom}{Hom}

\DeclareMathOperator{\Mat}{Mat}

\DeclareMathOperator{\Ima}{Im}

\DeclareMathOperator{\diag}{\mathrm{diag}}
\DeclareMathOperator{\Tr}{Tr}

\DeclareMathOperator*{\Alt}{Alt}

\DeclareMathOperator{\into}{\hookrightarrow}

\DeclarePairedDelimiter\abs{\lvert}{\rvert}

\DeclareMathOperator{\Gal}{Gal}
\DeclareMathOperator{\Out}{Out}

\DeclareMathOperator{\GL}{GL}
\DeclareMathOperator{\SL}{SL}

\DeclareMathOperator{\Sp}{Sp}

\DeclareMathOperator{\<}{\vartriangleleft}

\DeclareMathOperator{\ad}{ad}

\DeclareMathOperator{\mfgl}{\mathfrak{gl}}

\DeclareMathOperator{\mfg}{\mathfrak{g}}
\DeclareMathOperator{\mfn}{\mathfrak{n}}
\DeclareMathOperator{\mfh}{\mathfrak{h}}

\DeclareMathOperator{\mfsp}{\mathfrak{sp}}
\DeclareMathOperator{\htt}{ht}

\newcommand{\KMS}[1]{\mathscr{G}_{#1}}

\DeclareMathOperator{\HB22}{HB_2^{(2)}}
\DeclareMathOperator{\CC}{\mathcal{CC}} 
\DeclareMathOperator{\re}{re}
\DeclareMathOperator{\im}{im}
\DeclareMathOperator{\ma}{ma}
\DeclareMathOperator{\mfG}{\mathfrak{G}}
\DeclareMathOperator{\mfU}{\mathfrak{U}}

\DeclareMathOperator{\K}{\mathbb{K}}
\DeclareMathOperator{\Z}{\mathbb{Z}}
\DeclareMathOperator{\C}{\mathbb{C}}
\DeclareMathOperator{\N}{\mathbb{N}}

\DeclareMathOperator{\F}{\mathbb{F}}

\title[Lie type quotients of $\mathfrak{U}^+_{\mathrm{HB}_2^{(2)}}(\K)$]{Lie type quotients of the maximal unipotent subgroup of Kac-Moody groups of type $\mathrm{HB}_{2}^{(2)}$}
\author{Robynn Corveleyn}
\date{\today}
\thanks{The author is grateful for financial support from the FWO and the F.R.S.–FNRS under the Excellence of Science (EOS) program (project ID 40007542).}
\address{Robynn Corveleyn ({\tt robynn.corveleyn@uclouvain.be}) \newline Institut de Recherche en Math\'{e}matique et Physique, UCLouvain, Chemin du Cyclotron 2, 1348 Louvain-la-Neuve, Belgium.}
\begin{document}
\maketitle

\begin{abstract}
In this article, we construct infinite families $(G_n)_{n \in \N}$ of finite simple groups $G_n$ of Lie type, such that the rank of $G_n$ strictly increases as $n$ tends to infinity, and such that each $G_n$ is a quotient of the maximal unipotent subgroup $U^+$ of the (minimal) Kac-Moody group $\mathfrak{G}_{A}(\mathbb{K})$ of type $\mathrm{HB}_{2}^{(2)}$ over a finite field $\mathbb{K}$. Moreover, we show that the quotient maps lead to the construction of an infinite family of bounded degree spectral high-dimensional expanders. These provide the first class of examples of infinite families of high-dimensional expanders constructed from Lie type groups of unbounded rank.
\end{abstract}
\epigraph{Ik wil een Singer \newline wij willen een Singer \newline wij eisen een Singer \newline wat wij willen is ons recht}{— Paul Van Ostaijen \par \textit{Huldegedicht aan Singer}, 1921}
\section{Introduction}
\begingroup
\renewcommand\theproposition{\Alph{proposition}}
In this article, our aim is to construct arbitrarily large finite simple quotients of Lie type of the maximal unipotent subgroup $U^+ := \mfU^+_{{\tiny \HB22}}(\F_q)$ of a Kac-Moody group of type  $\HB22$ over a finite field $\F_q$. We then show that these lead to the construction of infinite families of bounded degree spectral high-dimensional expanders. 

Spectral higher dimensional expansion is a generalisation to simplicial complexes (of dimension at least 2) of the notion of spectral expansion in the theory of graphs. High-dimensional expanders satisfy strong local-to-global results, making them powerful and useful tools for many applications. For example, they have been used to unify the study of local testability of codes and that of quantum LDPC codes, as well as in questions about topological overlapping properties, see \cite{GotlibKaufman}. Not many explicit constructions of high-dimensional expanders are currently known. We contribute to this study by giving, to the best of our knowledge, the first construction of infinite families of high-dimensional expanders arising from coset complexes of Lie type groups of unbounded rank. In \cite{IngaPaper}, the authors show that finite quotients of \emph{Kac-Moody-Steinberg-groups} (or KMS-groups for short) can give rise to high-dimensional expanders. We show that the quotients of $U^+$ we construct in this paper can be used to apply their result.

\par Let $A$ be a generalised Cartan matrix (GCM) $(A_{ij})_{i,j \in I}$ in the sense of \cite[Chapter 1]{Kac}. A subset $J \subseteq I$ is called \emph{spherical} if the matrix $A_J := (A_{ij})_{i,j \in J}$ is a Cartan matrix. For $n \leq \abs{I}$, we say that $A$ is \emph{$n$-spherical} if every subset $J \subseteq I$ of size at most $n$ is spherical. For any $2$-spherical GCM and any field $\K$, one defines the KMS-group $\KMS{A}(\K)$ as an amalgamated product of the unipotent radicals $\mfU^+_{A_J}(\K)$ of the standard Borel subgroups of the rank $2$ Chevalley groups $\mfG_{A_J}(\K)$ of type $A_J$ over $\K$, for all $J\subseteq I$ such that $\abs{J} \leq 2$ (the precise definition is given in \cref{sect KMS groups}).

In this paper, we consider a GCM $A$ of type $\HB22$, as given in \Cref{HB22}. When $\K = \F_p$ with $p$ a prime, the KMS-group of type $A$ is given by the following presentation:
\begin{align}
\KMS{\HB22}(\F_p) \cong \langle a, b , c \mid &  \: \:  a^p, b^p, c^p, [a,b,a], [a,b,b] , [c,b,c], [c,b,b,b],[c,b,b,c], \\ & \: \: [c,a,c], [c,a,a,a], [c,a,a,c] \rangle. \notag
\end{align}
\begin{figure}[h]
\centering
\vspace{-24pt}
\caption{The GCM of type $\HB22$ and its corresponding Dynkin diagram.}
\label{HB22}
\begin{minipage}{.3\textwidth}
$A = \left(\begin{matrix}
2 & -1 & -2\\
-1 &  2 & -2 \\
-1 & -1 & 2
\end{matrix}\right)$
\end{minipage}
\hspace{1cm}
\begin{minipage}{.2\textwidth}
\begin{tikzpicture}[roundnode/.style={draw,shape=circle,fill,inner sep=2pt},decoration={markings,mark=at position 0.6 with {\arrow[]{>}}}]
\node[roundnode,label={south:$a $}] (b) {};
\node[roundnode,label={north:$c $}] (a) [above right = of b,xshift = -9pt] {}; 
\node[roundnode,label={south:$b $}] (c) [below right = of a,xshift = -9pt] {};
\draw[double, double distance = 2pt,postaction={decorate}] (a) -- (b);
\draw[-] (b) -- (c);
\draw[double, double distance = 2pt,postaction={decorate}] (a) -- (c);
\end{tikzpicture}
\end{minipage}
\end{figure}

This group was first studied in \cite{CCKW}. The authors show that it is a hyperbolic group (for $p$ odd) with property (T) (for $p \geq 7$) (\cite[Theorem 1.3]{CCKW}). Moreover, they explore its finite simple quotients, and show in particular that (for $p$ odd) the KMS-group $\KMS{\HB22}(\F_p)$ has a finite simple quotient containing an isomorphic copy of $\Alt(n)$, for every $n \geq 1$ (\cite[Corollary 7.19]{CCKW}).

We extend their work in several ways. For any field $\K$ with $\abs{\K} \geq 4$, there is a canoncial surjective group homomorphism \[
\Theta_{\K, \HB22} \colon \KMS{\HB22}(\K) \to \mfU^+_{\HB22}(\K).
\]
Inspired by the quotients of $\KMS{\HB22}(\F_p)$ described in \cite[Proposition 7.16]{CCKW}, we construct group homomorphisms
\[\Phi_{M_a,M_b,M_c} \colon \mfU^+_{\HB22}(\K) \to \GL_{4n}(\K),\]
for any field $\K$ and any integer $n \geq 1$, parametrised by matrices $M_a,M_b,M_c \in \Mat_{n \times n}(\K)$. We then determine the isomorphism type of these quotients explicitly, for suitable choices of matrices $M_a,M_b,M_c \in \Mat_{n \times n}(\K)$.

For $A$ of type $\HB22$ with root system $\Delta$, let $\mfg_A(\C)$ be the (derived) Kac-Moody algebra over $\C$ (see \cite[Chapter 1]{Kac}). By definition, $\mfg_A(\C)$ has a triangular decomposition $\mfn^+ \oplus \mfh' \oplus \mfn^-$, and the subalgebra $\mfn^+$ is generated by elements $e_a, e_b$ and $e_c$ corresponding to the positive simple roots $a,b,c \in \Delta$. 
Let $M_a, M_b, M_c \in \Mat_{n \times n}(\C)$. Then the assignment 
\begin{align*}
\phi_{M_a,M_b,M_c}^{[n]} \colon & \mfn^+ \to \mfgl_{4n}(\C) \\ 
& \: e_a \: \mapsto  M_a(E_{14}+E_{23}), \quad 
e_b \:  \mapsto M_b(E_{21}-E_{34}), \quad 
e_c\:  \mapsto  M_c E_{42},
\end{align*}
is a Lie algebra homomorphism (see \Cref{is Lie alg homomorphism} below). When $n = 1$ and $M_i \neq 0$ for all $i \in \{a,b,c\}$, it is a surjective homomorphism 
\[\phi_{M_a,M_b,M_c}^{[1]} \colon \mfn^+ \to \mfsp_4(\C),\]
since for $\Psi = \pm \{\alpha_1 , \alpha_2, \alpha_1+\alpha_2, 2 \alpha_1 +\alpha_2\}$ a root system of type $B_2$, the root spaces of $\mfsp_4(\C)$ corresponding to $-\alpha_1$, $-\alpha_2$ and $\alpha_1 + \alpha_2$ generate $\mfsp_4(\C)$ and are respectively spanned by $E_{21}-E_{34}$, $E_{42}$ and $E_{14}+E_{23}$.

We show in \Cref{prop:alg morphism lifts to group morphism} that for any $n \geq 1$ and any $M_a, M_b, M_c \in \Mat_{n \times n}(\C)$, the Lie algebra homomorphism $\phi_{M_a,M_b,M_c}^{[n]}$ induces an “exponentiated” group homomorphism \[\Phi_{M_a,M_b,M_c}\colon \mfU^+_A(\C) \to \GL_{4n}(\C).\] After passing to an arbitrary field, we obtain in particular for any field $\K$ and any matrices $M_a,M_b,$ and $ M_c  \in \Mat_{n \times n}(\K)$ an induced group homomorphism \[\Phi_{M_a,M_b ,M_c} \colon \mfU^+_A(\K) \to \GL_{4n}(\K).\]

When $\abs{\K} \geq 4$, the group $\mfU^+_{A}(\K)$ is generated by three copies of $(\K,+)\cong \mfU_i(\K)\leqslant \mfU^+_{A}(\K)$, parametrised by the isomorphisms $\K \to \mfU_i(\K), \lambda \mapsto x_i(\lambda)$, for $i \in \{a,b,c\}$ (see \Cref{U+ generated by simple root subgroups} below). Our first result can then be formulated more precisely as follows.
\begin{theorem}
\label{group hom to SL4K}
Let $\K$ be a field with $\abs{\K} \geq 4$ and let $n \geq 1$. Suppose  $M_a, M_b, M_c \in \Mat_{n \times n}(\K)$. 
For all $\lambda \in \K$,  set
\begin{equation*}
V_a(\lambda)  := \left( \begin{matrix}
1 & 0 & 0 & \lambda M_a \\ 0 & 1 & \lambda M_a & 0 \\ 0 & 0 & 1 & 0 \\ 0 & 0 & 0 & 1
\end{matrix}\right) ,
V_b(\lambda) := \left( \begin{matrix}
1 & 0 & 0 & 0 \\ \lambda M_b & 1 & 0 & 0 \\ 0 & 0 & 1 & -\lambda M_b \\ 0 & 0 & 0 & 1
\end{matrix}\right),  
V_c(\lambda)  := \left( \begin{matrix}
1 & 0 & 0 & 0 \\ 0 & 1 & 0 & 0 \\ 0 & 0 & 1 & 0 \\ 0 & \lambda M_c & 0 & 1
\end{matrix} \right).
\end{equation*}
Then the assignment \[ x_a(\lambda) \mapsto V_a(\lambda), \quad x_b(\lambda) \mapsto V_b(\lambda), \quad x_c(\lambda) \mapsto V_c(\lambda), \] 
induces a group homomorphism \[\Phi_{M_a,M_b,M_c} \colon \mfU^+_{\HB22}(\K) \to \SL_{4n}(\K).\]
Moreover, if $M_a, M_b$ and $M_c$ are invertible, then $\Phi_{M_a,M_b,M_c}$ is injective on the rank $2$ subgroups $\mfU^+_{A_J}(\K)$  ($\abs{J} = 2$) of $\mfU^+_{\HB22}(\K)$.
\end{theorem}
For $\K = \F_p$, the composition of the above group homomorphism with the canonical surjective homomorphism $\Theta_{\K,\HB22}$ is precisely the map from \cite[Proposition 7.15]{CCKW}.

The bulk of our paper is devoted to computing the images of the group homomorphisms given by \Cref{group hom to SL4K} explicitly, for suitable choices of $M_a,M_b,M_c \in \Mat_{n \times n}(\K)$, with $\K = \F_{q}$ a finite field and $n \geq 1$. Note that, by construction, if $M_a, M_b , M_c \in \GL_1(\F_q) = (\F_q)^{\times}$, the image is precisely $\Sp_4(\F_q)$. We construct matrices $M_a, M_b, M_c \in \GL_{n}(\F_q)$ with $n > 1$, such that the image is isomorphic to $\Sp_{4n}(\F_q)$. Surprisingly, we are also able to give constructions of matrices $M_a,M_b, M_c \in \GL_{n}(\F_q)$ such that the image of $\Phi_{M_a,M_b,M_c}$ is the whole of $\SL_{4n}(\F_q)$. More precisely, our main theorem is the following.

\begin{theorem}
\label{main theorem}
Let $p > 2$ and $k > 3$ be distinct primes. Let $q = p^r > 3$ for some $r \geq 1$. Then there exist matrices $M_a, M_b,M_c \in \GL_k(\F_q)$ and symmetric matrices $M_a',M_b',M_c' \in \GL_k(\F_q)$ such that 
\[
\Phi_{M_a,M_b,M_c} \colon \mfU^+_{\HB22}(\F_q) \to \SL_{4k}(\F_q)
\]
and
\[
\Phi_{M_a',M_b',M_c'}\colon \mfU^+_{\HB22}(\F_q) \to \Sp_{4k}(\F_q)
\]
are surjective group homomorphisms. 
\end{theorem}
The proof of \Cref{main theorem}, as well as explicit constructions of the matrices $M_a,M_b,M_c \in \GL_k(\F_q)$ and $M_a',M_b',M_c' \in \GL_k(\F_q)$, can be found in \cref{sect: SL4K,sect: SP4k}.

Finally, we show in \cref{sect:HDX} that the quotients in \Cref{main theorem} satisfy the hypotheses from \cite[Theorem 4.3]{IngaPaper}, thus providing new infinite families of bounded degree, spectral high-dimensional expanders. 
\begin{corollary}
\label{col HDX}
Let $p > 2$ be a prime and $q = p^r  > 3 $ with $r \geq 1$. For all primes $k > 3$  such that $p \neq k$, there exist matrices $X_k, Y_k, Z_k \in \SL_{4k}(\F_q)$ such that, defining 
\[H_{T_k} := \langle  \{X_k,Y_k,Z_k\} \setminus \{T_k\} \rangle \leqslant \SL_{4k}(\F_q), \quad T_k \in \{X_k,Y_k,Z_k\}\]
the coset complex 
\[\mathcal{X}_{q,k} := \CC \left(\SL_{4k}(\F_q), \{H_{T_k}\}_{T_k \in \{X_k,Y_k,Z_k\}}\right)\]
is a $\frac{\sqrt{2q}+2}{q-2}$-spectral high-dimensional expander of dimension $2$. 

Similarly, for all primes $k > 3$  such that $p \neq k$ there exist matrices $X'_k,Y'_k,Z'_k \in \Sp_{4k}(\F_q)$ such that the coset complex
\[\mathcal{X}'_{q,k} := \CC\left(\Sp_{4k}(\F_q), \{H_{T_k}\}_{T_k \in \{X'_k,Y'_k,Z'_k\}}\right)\]
is a $\frac{\sqrt{2q}+2}{q-2}$-spectral high-dimensional expander of dimension $2$.
In partciular, $(\mathcal{X}_{q,k})_{k \in P}$ and $(\mathcal{X}'_{q,k})_{k \in P}$  are infinite families of bounded degree, spectral high-dimensional expanders, where $P := \{k \text{ prime } \mid p \neq k \text{ and } k > 3\}$.
\end{corollary}

This article is structured as follows. In \cref{sect:preliminaries}, we introduce the preliminary definitions and results we will need, and fix some notation. \Cref{sect:integration} is devoted to the proof of the first part of \Cref{group hom to SL4K}. Then in \cref{sect: SL4K,sect: SP4k}, we construct quotients of $\KMS{\HB22}(\F_q)$ respectively isomorphic to $\SL_{4k}(\F_q)$ and $\Sp_{4k}(\F_q)$, which together prove \Cref{main theorem}. Finally, in \cref{sect:HDX} we prove the second part of \Cref{group hom to SL4K}, as well as \Cref{col HDX}.

\subsection*{Acknowledgements}
I am grateful to Timothée Marquis for his guidance and the many enriching and fruitful discussions, and to Pierre-Emmanuel Caprace for his insights on his paper \cite{CCKW}. I also thank François Arnault for sharing the source code to \cite{ArnaultF} with me, which facilitated preliminary calculations with Singer elements.

\endgroup
\setcounter{proposition}{0}
\section{Preliminaries}
\label{sect:preliminaries}

\subsection*{Notation}
In this text, for any two elements $g, h$ in some group $G$, we define their commutator as $[g,h] := g^{-1}h^{-1}gh$. A nested commutator will be written as $[g_1, \ldots, g_n] := [[g_1, \ldots,g_{n-1}],g_n]$. 

We let $E_{i,j}$ denote the matrix with entries $1$ at position $(i,j)$ and $0$ elsewhere. The size of $E_{i,j}$ will usually be clear from the context; when ambiguity is possible, we will specify its size explicitly. For a $k \times k$-matrix $B$, the notation $B E_{i,j}$ with $i,j \in \{1,\ldots, n\}$ denotes the $n k \times nk$ block-matrix $(C_{s,t})_{1 \leq s,t \leq n}$, where $C_{s,t}$ is the $k \times k$ zero-matrix if $(s,t) \neq (i,j)$, and $C_{i,j} = B$.

For any unital ring $R$, we denote by $R^{\times}$ its subset of invertible elements. We denote by $\N$ the set of non-negative integers.
\subsection{Kac-Moody algebras}
Let $A = (A_{ij})_{i,j \in I}$ be a matrix with integer coefficients. It is called a \emph{generalised Cartan matrix} (abbreviated to GCM throughout the remainder of this paper) if  $A_{ii} = 2$, $A_{ij} \leq 0$ and $A_{ij} = 0 \iff A_{ji} = 0$, for all $i \neq j \in I$. A GCM is called \emph{spherical} (equivalently, of \emph{finite type}) if it is a Cartan matrix.

A GCM $A$ is said to be \emph{$n$-spherical} if $A_J := (A_{ij})_{i,j \in J}$ is spherical, for all $J \subseteq I$ with $\abs{J} \leq  n$. 
The size $\abs{I}$ of $I$ is called the \emph{rank} of $A$. \par 
To any GCM, there is an associated complex Lie algebra $\mfg_A(\C)$, with a presentation defined by generators $e_i, f_i, \alpha_i^{\vee}$, with relations as given by a \emph{Serre presentation} (see for example \cite[Definition 3.17]{boekTimothée}). It is called the (derived) \emph{Kac-Moody algebra of type $A$ over $\C$}. In case $A$ is of finite type, this is precisely the finite-dimensional complex semisimple Lie algebra of type $A$.

We denote by $\mfn^+_A(\C)$ (respectively, $\mfn^-_A(\C)$) the subalgebra generated by the elements $e_i$ (respectively $f_i$). 
It follows from the Serre presentation that $\mfn^+_A(\C)$ has the presentation
\begin{equation}
\label{Serre presentation}
\mfn_A^+(\C) = \left\langle e_i \mid \ad(e_i)^{-A_{ij}+1}e_j = 0, \quad i\neq j \in I \right\rangle_{\C}.
\end{equation}
The Kac-Moody algebra $\mfg_A(\C)$ has a triangular decomposition $\mfg_A(\C) := \mfn^+_A(\C) \oplus \mfh' \oplus \mfn^-_A(\C)$, where $\mfh'$ is the subalgebra spanned by the elements $\alpha_i^{\vee}$.

Let $Q := \bigoplus_{i \in I} {\Z}{\alpha_i}$ be the free abelian group generated by symbols $\alpha_i$. They are called the \emph{simple roots}. Moreover let $Q^{\pm} := \pm \bigoplus_{i \in I} {\N}{\alpha_i}$.  The Lie algebra $\mfg_A(\C)$ admits a $Q$-gradation
\[
\mfg_A(\C) = \bigoplus_{\alpha \in Q^+} \mfg_{\alpha} \oplus \mfh' \oplus \bigoplus_{\alpha \in Q^-} \mfg_{\alpha},\]
obtained by defining $\deg(e_i) := \alpha_i$ and $\deg(f_i) := -\alpha_i$. 

An element $\alpha \in Q \setminus \{0\}$ is called a \emph{root} if $\mfg_{\alpha} \neq \{0\}$. We write $\Delta$ for the set of roots, and define $\Delta_{\pm} := \Delta \cap Q^{\pm}$. Then $\Delta$ decomposes as $\Delta = \Delta_+ \sqcup \Delta_-$, respectively called the set of \emph{positive} and \emph{negative} roots. The \emph{height} of a root $\gamma = \pm \sum_{i \in I} n_i \alpha_i$ for $n_i \in \N$, is defined as \[\htt(\gamma) := \pm \sum_{i \in I} n_i.\]

We define a subgroup $W \leqslant \GL(Q)$, generated by the \emph{fundamental reflections}
\[
s_i \colon Q \to Q, \quad \alpha_j \mapsto \alpha_j - A_{ij}\alpha_i.
\]
Then $W$ is called the \emph{Weyl group} of $\Delta$. A root $\alpha$ in the $W$-orbit $W \cdot \{\alpha_1, \ldots, \alpha_n\}$ is called a \emph{real root}, while any $\delta \in \Delta \setminus W\cdot \{\alpha_1, \ldots, \alpha_n\}$ is called an \emph{imaginary root}. The corresponding subsets of $\Delta$ are respectively denoted by $\Delta^{\re}$ and $\Delta^{\im}$. We also set $\Delta^{\re}_{\pm} := \Delta^{\re} \cap \Delta_{\pm}$.

\subsection{A \texorpdfstring{$\Z$}{Z}-form of the universal enveloping algebra of \texorpdfstring{$\mfn^+$}{n+}}

For any complex associative algebra $B$, a \emph{$\Z$-form of $B$} is a subring $B_{\Z} \subseteq B$ such that $B_{\Z} \otimes \C \cong B$. For any GCM $A$, we denote by $\mathcal{U}^+_{\C}$ the universal enveloping algebra of $\mfn_A^+(\C)$. As shown in \cite[Proposition 7.4]{boekTimothée}, it has a $\Z$-form $\mathcal{U}_{\Z}^+$ given by its subring generated by the elements $e_i^{(s)}$, for all $i \in I$ and  $s \in \N$, where for any $u \in \mathcal{U}_{\C}^+$, 
\[
u^{(s)} := \frac{1}{s!} u^s.
\] For a field $\K$, we define $\mfn_A^+(\K) := (\mathcal{U}^+_{\Z} \cap \mfn_A^+(\C)) \otimes \K$ and $\mathcal{U}^+_{\K} := \mathcal{U}^+_{\Z} \otimes \K$.

The ring $\mathcal{U}_{\Z}^+$ is graded with respect to the height of the roots, i.e.\
$\mathcal{U}_{\Z}^+ = \bigoplus_{n \geq 0} \mathcal{U}_{n, \Z}^+$,
where each $\mathcal{U}_{n,\Z}^+$ is the submodule generated by the products $e_{i_1}^{(s_1)} \ldots e_{i_t}^{(s_t)}$, such that $s_1 + \ldots + s_t = n$. The \emph{completion} of $\mathcal{U}_{\Z}^+$ with respect to this gradation is given by 
\[
\widehat{\mathcal{U}}^+_{\Z} := \prod_{n \geq 0} \: \mathcal{U}^+_{n,\Z}.
\]
For a field $\K$, we then define $\widehat{\mathcal{U}}^+_{\K} := \prod_{n \geq 0} \left(\mathcal{U}^+_{n,\Z} \otimes \K\right)$.

\subsection{Kac-Moody groups}
\label{subsec Kac-Moody grps}
Let $A = (A_{ij})_{i,j \in I}$ be a GCM with associated root system $\Delta$, with simple roots $\{\alpha_i \mid i \in I\}$ and Weyl group $W$. We denote by $\mfG_A$ the \emph{constructive Tits functor of type $A$ of simply connected type}, see \cite{Tits87} or \cite[Definition 7.47]{boekTimothée}. Let $\K$ be a field. Then $\mfG_A(\K)$ is called the \emph{minimal Kac-Moody group of type $A$ over $\K$}. It is an amalgamated product of \emph{root subgroups} $\mfU_{\alpha}(\K) \cong (\K,+)$, indexed by the real roots $\alpha \in \Delta^{\re}$. We denote by $x_{\alpha} \colon (\K, +) \to \mfU_{\alpha}(\K)$ the corresponding group isomorphism. The group $\mfG_A(\K)$ is generated by the subgroups $\mfU_{\pm \alpha_i}(\K)$ associated to the simple roots and their opposites. 

When $A$ is of finite type, $\mfG_A(\K)$ coincides with the (simply connected) \emph{Chevalley group of type $A$ over $\K$} (see \cite[Exercise 7.50]{boekTimothée}). \par 
More precisely, if $A$ is of type $A_n$, with $n > 1$, then $\mfG_A(\K)$ is isomorphic to $\SL_{n+1}(\K)$, where the isomorphism is given by the identification 
\[
x_{\alpha_i}(\lambda) := I_{n+1} + \lambda E_{i, i+1}, \quad x_{-\alpha_i}(\lambda) := I_{n+1} + \lambda E_{i+1,i}, \qquad \forall \lambda \in \K, 1 \leq i \leq n.
\]
In case $A$ is of type $C_n$, then $\mfG_A(\K)$ is isomorphic to $\Sp_{2n}(\K)$, where the identification is given by
\begin{align*}
x_{\alpha_i}(\lambda) := \left(\begin{matrix}
I_{n} + \lambda E_{i,i+1} & 0 \\ 
0 & I_{n} - \lambda E_{i+1,i}
\end{matrix}\right), & \quad \text{ for } 1 \leq i \leq n-1, \\
 x_{-\alpha_i}(\lambda) := \left(\begin{matrix}
I_{n} + \lambda E_{i+1,i} & 0 \\ 0 & I_{n} - \lambda E_{i,i+1}
\end{matrix}\right), &  \quad \text{ for } 1 \leq i \leq n-1,
\end{align*}
and 
\[
x_{\alpha_n}(\lambda) := \left(\begin{matrix}
I_n & \lambda E_{n,n} \\
0 & I_n
\end{matrix}\right), \quad   x_{-\alpha_n}(\lambda) := \left(\begin{matrix}
I_n & 0 \\
\lambda E_{n,n} & I_n
\end{matrix}\right),
\]
for all $\lambda \in \K$. 

A pair of roots $\alpha, \beta \in \Delta^{\re}$ is said to be \emph{prenilpotent} if there exist some elements  $w , w' \in W$ such that $w\cdot \{\alpha,\beta\} \subseteq \Delta_+$ and $w' \cdot \{\alpha, \beta\} \subseteq \Delta_-$. For any prenilpotent pair of roots $\{\alpha,\beta\}$, the set $\left({\N}{\alpha} + {\N}{\beta}\right) \cap \Delta$ is finite and contained in $\Delta^{\re}$. The following relations hold in $\mfG_A(\K)$, for any $\lambda, \mu \in \K$ and any prenilpotent pair of roots $\{\alpha, \beta\}$:
\begin{equation}
\label{commutator relations chevalley}
[x_{\alpha}(\lambda),x_{\beta}(\mu)] = \prod_{\substack{\gamma = i \alpha + j\beta \in \Delta \\ i,j \in \N_{> 0}}} x_{\gamma}(C_{i,j}^{\alpha\beta} \lambda^{i} \mu^{j}),
\end{equation}
where $C_{i,j}^{\alpha\beta}$ are integers which may be computed explicitly, see for example \cite[Exp. XXIII, §3]{SGA} for the spherical rank 2 case.

For any GCM $A$, we set 
\[\mfU_A^+(\K) := \left\langle \mfU_{\alpha}(\K) \mid \alpha \in \Delta^{\re}_+\right\rangle \leqslant \mfG_A(\K).\]
We record the following fact. 
\begin{lemma}[{\cite[§1, Théorème]{AbramenkoMuhlherr}}]
\label{U+ generated by simple root subgroups}
Let $A$ be a 2-spherical GCM, with associated root system $\Delta$. Suppose $\K$ is a field containing at least 4 elements. Then $\mfU_A^+(\K)$ is generated by the root subgroups $\mfU_{\alpha_i}(\K)$ associated to the simple roots $\alpha_i$ of $\Delta$.
\end{lemma}
In the spherical case, elements of $\mfU^+_A(\K)$ moreover have the following normal form.
\begin{lemma}[{\cite[Theorem 5.3.3]{Carter}}]
\label{unique expression U+}
Let $A$ be a Cartan matrix with associated root system $\Delta$. Fix a total order $<$ on $\Delta_+$. Then each element $u \in \mfU^+_A(\K) $ has a unique expression of the form
\[
u = \prod_{\alpha \in \Delta_+} x_{\alpha}(\lambda_{\alpha}),
\]
where the product is taken over the roots in increasing order. 
\end{lemma}

Let $\mfU^{\ma +}_A$ be the affine group scheme associated to a GCM $A$, as defined in \cite[§8.5.1-§8.5.2]{boekTimothée}. For any field $\K$, the group $\mfU^{\ma +}_A(\K)$ can be identified with a subgroup of the multiplicative group of invertible elements in $\widehat{\mathcal{U}}^+_{\K}$ (see \cite[Theorem 8.51 (3)]{boekTimothée}).
\begin{lemma}[{\cite[Corollary 8.75]{boekTimothée}}]
\label{inclusion U+ in Uma+}
Let $\K$ be a field. Then for any GCM $A$, there is an injective group homomorphism
\[\mfU^+_A(\K) \into \mfU^{\ma +}_A(\K).\]
\end{lemma}
Under the above inclusion map, the elements $x_{\alpha_i}(\lambda) \in \mfU^+_A(\K)$ are mapped to the elements \[\exp(\lambda e_i) := \sum_{s \in \N} e_i^{(s)} \otimes \lambda^s \in \widehat{\mathcal{U}}^+_{\K}.\]

\subsection{Singer elements}
Throughout this section, we fix an integer $k \geq 1$ and a prime power $q = p^r$.
\begin{lemma}[{\cite[Satz 3.10 \& Satz 7.3]{Huppert}}]
\label{maximal order GL}
The elements of maximal order in $\GL_k(\F_q)$ have order $q^k-1$.
\end{lemma}

An element $S \in \GL_k(\F_q)$ which has order $q^k-1$ is called a \emph{Singer element}.

\begin{lemma}
\label{action is irreducible}
\label{Singer not block diag}
A Singer element $S \in \GL_k(\F_q)$ acts irreducibly on $\F_q^k$ via multiplication. In particular, a Singer element cannot be a block-diagonal matrix.
\end{lemma}
\begin{proof}
Let $\F_q^k \cong V_1 \oplus \ldots \oplus V_n$ be the decomposition of $\F_q^k$ into irreducible $\langle S \rangle$-representations provided by Maschke's theorem, where the action of the Singer element $S \in \GL_k(\F_q)$ on $\F_q^k$ is given by matrix multiplication. Then by \Cref{maximal order GL} the order of $S|_{V_i}$ is at most $q^{d_i} -1$, with $d_i := \dim_{\F_q}(V_i)$. If $n >1$, then 
\begin{align*}
o(S) = q^k-1 & \leq \mathrm{lcm}\left\{o(S|_{V_i})  \mid 1 \leq i \leq n\right\} \\
& \leq \prod_{i = 1}^n (q^{d_i}-1) \leq (q^{d_1} -1) (q^{d_2}-1)q^{k-d_1-d_2}.
\end{align*}
But the latter inequality is equivalent to 
\[q^{d_1+d_2-k} \geq q^{d_1}+q^{d_2}-1,\]
a contradiction.
\end{proof}

The following fact is essential to our work below. 
\begin{proposition}[{\cite[II, Satz 7.3]{Huppert}}]
\label{centraliser of Singer element}
For any Singer element $S \in \GL_k(\F_q)$, the subgroup generated by $S$ is its own centraliser in $\GL_{k}(\F_q)$.
\end{proposition}

\subsection{Overgroups of field extension subgroups}
Let $\K$ be a field. Suppose $f \colon V \to V$ is an endomorphism of an $n$-dimensional $\K$-vector space. Let $\mathcal{B}$ be a $\K$-basis of $V$. Then we denote by $[f]_{\mathcal{B}}$ the $n\times n$-matrix of $f$ with respect to the basis $\mathcal{B}$.

A \emph{symplectic form} is a non-degenerate, alternating bilinear form \[f \colon \K^{2n} \times \K^{2n} \to \K.\]
For a symplectic form $f$ and a fixed $\K$-basis $\mathcal{B} := \{e_1, \ldots, e_{2n}\}$ of $\K^{2n}$, we define its associated matrix $\Omega_f := (f(e_i,e_j))_{1 \leq i,j \leq 2n}$. The \emph{symplectic group of $f$} is given by 
\[
\Sp_{2n}(\K, \Omega_f) := \left\{ X \in \GL_{2n}(\K) \mid X^t \Omega_f X = \Omega_f\right\}.
\]

The \emph{standard symplectic form of rank $n$} is given by the symplectic form defined by the $2n \times 2n$-matrix \[\Omega_n := \left(\begin{matrix}
0 & -I_n \\ I_n & 0 
\end{matrix}\right).\]
The associated \emph{standard symplectic group} will be denoted by $\Sp_{2n}(\K) := \Sp_{2n}(\K,\Omega_n )$. A Gramm-Schmidt process shows that for any symplectic form $f$, there exists a basis $\mathcal{B}$ such that the associated matrix $\Omega_f$  equals $\Omega_n$. In particular, for all symplectic forms $f \colon \K^{2n} \times \K^{2n} \to \K$, the groups $\Sp_{2n}(\K,\Omega_f)$ are isomorphic.

If $n = 1$, then every symplectic form $f \colon \K^2 \times \K^2 \to \K$ is given by a matrix 
\[\Omega_f := \left(\begin{matrix}
0 & -x \\ x & 0
\end{matrix}\right), \]
for some $x \in \K^{\times}$. It follows that there is precisely one group $\Sp_{2}(\K,\Omega_f)$, which coincides with $\SL_2(\K)$.

Suppose $\mathbb{L}/\mathbb{K}$ is a field extension of degree $k$. Then $\mathbb{L}$ is a $k$-dimensional vector space over $\K$. For any $x \in \mathbb{L}$, we denote by $\mu_x \colon \mathbb{L } \to \mathbb{L}$ the $\K$-linear map defined by $\mu_x(y) := xy$ for all $y \in \mathbb{L}$.

Let $m \geq 1$. Any subgroup $G \leqslant \GL_{m}(\mathbb{L})$ can be embedded as a subgroup of $\GL_{km}(\K)$ by considering the action of $g \in G$ on the $km$-dimensional $\K$-space ${\mathbb{L}}^m$. Explicitly, one fixes a $\K$-basis $\mathcal{B} := \{b_1, \ldots, b_k \}$ of $\mathbb{L}$. The inclusion $G \into \GL_{km}(\K)$ is then given by\[(x_{i,j})_{1 \leq i,j \leq m} \mapsto ([\mu_{x_{i,j}}]_{\mathcal{B}})_{1 \leq i,j \leq m}.\]

Below we will consider groups $X \leqslant \SL_{2k}(\K)$, such that $X$ contains an isomorphic copy of $\SL_2(\mathbb{L})$, with the embedding given by the fixed choice of basis $\mathcal{B}$ of $\mathbb{L}$ over $\K$.

Let $m \geq 2$ be even. Denote by $\Psi \colon \K^k \to \mathbb{L}$ the vector space isomorphism given by the $\K$-basis $\mathcal{B}$. For a symplectic form $f \colon \mathbb{L}^m \times \mathbb{L}^m \to \mathbb{L}$ and a $\K$-linear homomorphism $\sigma \colon \mathbb{L} \to \K$, there is an associated symplectic form $B_{f,\sigma}$, given by the commutative diagram
\begin{equation}
\label{definitie B_f}
\begin{tikzcd}
\mathbb{L}^m \times \mathbb{L}^m \arrow[r, "f"] & \mathbb{L} \arrow[d, "\sigma"] \\
\K^{km} \times \K^{km} \arrow[u, "\Psi^{\times m} \times \Psi^{\times m}"] \arrow[r,"B_{f,\sigma}"] & \K
\end{tikzcd}
\end{equation}

Denote by $\{e_1, \ldots, e_m\}$ the standard $\mathbb{L}$-basis of $\mathbb{L}^m$. Then \[\Omega_{B_{f,\sigma}} = (B_{f,\sigma}(b_{i_1}e_{j_1},b_{i_2}e_{j_2}))_{1 \leq i_{\ell} \leq k, 1 \leq j_{\ell}\leq m}\] is the matrix associated to the symplectic form $B_{f,\sigma}$ defined in \cref{definitie B_f}.

In the case of finite fields, with $\K = \F_q$ and $\mathbb{L} = \F_{q^k}$ (for $q$ a prime power) and $k \geq 1$, the $\K$-linear homomorphisms $\sigma \colon \mathbb{L} \to \K$ have the following description:
\begin{equation}
\label{Hom Fqk Fq}
\Hom_{\F_q}(\F_{q^k}, \F_q) = \left\{\phi_x := \Tr_{{\F_{q^k}}/{\F_q}} \circ \mu_x \mid x \in \F_{q^k}\right\},
\end{equation}
with 
\[\Tr_{{\F_{q^k}}/{\F_q}} \colon \F_{q^k} \to \F_q, \quad x \mapsto \sum_{\tau \in \Gal({\F_{q^k}}/{\F_q})} \tau(x).\]

We can now formulate the following consequence of \cite[Theorem 1]{Li}.
\begin{proposition}
\label{refinement Li}
Let $p > 2$ and $k$ be primes, and let $q = p^r$ for some $r \geq 1$. Let $N$ be the image of the embedding of $\SL_2(\F_{q^k}) = \Sp_2(\F_{q^k},\Omega_1)$ in $\SL_{2k}(\F_q)$ for a fixed choice of $\F_q$-basis $\mathcal{B} = \{b_1, \ldots, b_k\}$ of $\F_{q^k}$. Let $X \leqslant \SL_{2k}(\F_q)$ be a subgroup containing $N$. Then precisely one of the following holds:
\begin{enumerate}[(i)]
\item $N \< X$,
\item $X = \SL_{2k}(\F_q)$,
\item There exists an $x \in \F_{q^k}^{\times}$ such that $\Sp_{2k}(\F_q,\Omega_{B_{f,\phi_x}}) \< X$,  where the symplectic form $B_{f,\phi_x}$ is as defined in \cref{definitie B_f} and $f \colon \F_{q^k}^2 \times \F_{q^k}^2 \to \F_{q^k}$ is the standard symplectic form defined by $\Omega_f = \Omega_1$.
\end{enumerate}
\end{proposition}

The matrix $\Omega_{B_{f,\phi_x}}$ in the above proposition is by definition given by 
\begin{equation}
\Omega_{B_{f,\phi_x}} = \left(
\begin{matrix}
0 & (-\Tr(x b_i b_j))_{1 \leq i,j \leq k} \\ (\Tr(x b_i b_j))_{1 \leq i,j \leq k} & 0
\end{matrix}
\right), \text{ where } \Tr := \Tr_{{\F_{q^k}}/{\F_q}}.
\end{equation}
\begin{proof}[Proof of \Cref{refinement Li}]
Since $k$ is prime by assumption, there are no intermediate fields $\K$ strictly between $\F_{q^k}$ and $\F_{q}$. Then \cite[Theorem 1]{Li} implies that precisely one of (i) or (ii) holds, or that there is some $\sigma \in \Hom_{\F_q}(\F_{q^k},\F_q)$ such that $\Sp_{2k}(\F_q,B_{f,\sigma})) \< X$, with $\Omega_f = \Omega_1$. The result then follows by \cref{Hom Fqk Fq}.
\end{proof}

The following is another immediate consequence of \cite[Theorem 1]{Li}.
\begin{proposition}
\label{normaliser of SL2}
Let $p > 2$ and $k$ be primes, and let $q = p^r$ for some $r \geq 1$. Let $N$ be the image of the embedding of $\SL_2(\F_{q^k})$ in $\GL_{2k}(\F_q)$ for a fixed choice of $\F_q$-basis of $\F_{q^k}$. Then the normaliser $N_{\GL_{2k}(\F_q)}(N)$ of $N$ in $\GL_{2k}(\F_q)$ is isomorphic to $\GL_2(\F_{q^k})\rtimes \Gal({\F_{q^k}}/{\F_q})$. Under this isomorphism, the inclusion of $N$ in $N_{\GL_{2k}(\F_q)}(N)$ corresponds to the natural inclusion of $\SL_{2}(\F_{q^k})$ as a subgroup of the first factor of the semidirect product $\GL_2(\F_{q^k})\rtimes \Gal({\F_{q^k}}/{\F_q})$.
\end{proposition}
\begin{proof}
It suffices to take $X := N_{\GL_{2k}(\F_q)}(N)$ in the statement of \cite[Theorem 1]{Li}.
\end{proof}

\section{Integrating group homomorphisms from \texorpdfstring{$\mfn^+$}{n+} to \texorpdfstring{$\mfU^+_A(\K)$}{U+}}
\label{sect:integration}
Throughout this section, let $A$ be a $2$-spherical GCM. Set $\ell = \abs{I}$ and let $e_1, \ldots, e_{\ell}$ be the canonical generators of $\mfn^+ := \mfn^+_A(\C)$. 
We consider Lie algebra morphisms
\[\phi \colon \mfn^+ \to B,\]
where $B$ is a unital associative $\C$-algebra, which is a Lie algebra for the \emph{additive commutator} $[x,y] := xy - yx$, for all $x, y\in B$. By the universal property of the universal enveloping algebra, there is an associated morphism of unital associative $\C$-algebras $\Phi \colon \mathcal{U}^+_{\C} \to B$.

The following then allows to integrate certain Lie algebra morphisms on $\mathfrak{n}^+$ to $\mfU^+_A(\K)$.
\begin{proposition}
\label{prop:alg morphism lifts to group morphism}
Let $\phi \colon \mfn^+ \to B$ be a  Lie algebra morphism, with associated algebra morphism $\Phi \colon \mathcal{U}^+_{\C} \to B$. Suppose that for every $i \in I$, the element $\phi(e_i) \in B$ is nilpotent. 
Suppose moreover that there is a $\Z$-form $B_{\Z}$ of $B$ such that $\Phi(\mathcal{U}_{\Z}^+) \subseteq B_{\Z}$. Let $\K$ be a field, and write $\Phi_{\K} \colon \mathcal{U}^+_{\K} \to B_{\Z} \otimes \K$ for the induced $\K$-algebra morphism. Then there is a subalgebra $\widetilde{\mathcal{U}}^+_{\K} \subseteq \widehat{\mathcal{U}}^+_{\K}$ containing $\mathcal{U}^+_{\K}$ and satisfying the following properties: 
\begin{enumerate}[(i)]
\item $\Phi_{\K}$ extends to  $\widetilde{\mathcal{U}}^+_{\K}$.
\item If $\abs{\K} \geq 4$, then \[\mfU^+_A(\K) \subseteq \left(\widetilde{\mathcal{U}}^+_{\K}\right)^{\times}.\] 
\end{enumerate}
In particular, for any field $\K$ with $\abs{\K} \geq 4$, there is an associated group homomorphism
\[\Phi_{\K} \colon \mfU^+_A(\K) \to (B_{\Z} \otimes \K)^{\times},\qquad x_{\alpha_i}(\lambda) \mapsto \sum_{n \in \N} \frac{1}{n!} \phi(e_i)^n \otimes \lambda^n.\]
\end{proposition}
\begin{proof}
Let $\phi \colon \mfn^+ \to B$ be as in the statement, and consider the associated morphism of algebras $\Phi \colon \mathcal{U}_{\C}^+ \to B$. For $\K$ a field, define \[\Phi_{\K} \colon \mathcal{U}^+_{\K} \to B_{\Z} \otimes \K, \quad \textstyle\sum_{i} u_i \otimes \lambda_i \mapsto \textstyle\sum_{i}\Phi(u_i) \otimes \lambda_i.\]
Moreover, we define the following subset of $\widehat{\mathcal{U}}_{\K}^+$.
\[
\widetilde{\mathcal{U}}^+_{\K} := \left\{ u = \sum_{n \geq 0} u_n \in \widehat{\mathcal{U}}^+_{\K} \mid u_n \in \mathcal{U}^+_{n,\Z} \otimes \K, \quad \exists \ell_u \geq 0  \text{ such that } \Phi_{\K}(u_m) = 0, \: \forall m \geq \ell_u \right\}.
\]
Then $\widetilde{\mathcal{U}}^+_{\K}$ is a subalgebra of $\widehat{\mathcal{U}}^+_{\K}$. Indeed, it is trivially closed under sums by linearity of $\Phi_{\K}$. Moreover, products in $\widehat{\mathcal{U}}^+_{\K}$ are given by 
\[
(u \cdot v)_m =  \sum_{k = 0}^{m} u_k v_{m-k},
\]
for any $u, v  \in \widehat{\mathcal{U}}^+_{\K}$, where $x_m$ denotes the $m$-th homogeneous component of an element $x \in \widehat{\mathcal{U}}^+_{\K}$.
It follows that if $u , v\in \widetilde{\mathcal{U}}^+_{\K}$, with $\ell_u$ and $\ell_v$ such that $\Phi_{\K}(u_k) = 0$ for $k \geq \ell_u$ and $\Phi_{\K}(v_k) = 0$ for $k \geq \ell_v$, then 
\[\Phi_{\K}((u\cdot v)_m) = \sum_{k = 0}^m \Phi_{\K}(u_k) \Phi_{\K}(v_{m-k}) = 0, \text{ for all } m \geq \ell_u+\ell_v,\]
and hence $\widetilde{\mathcal{U}}^+_{\K}$ is closed under multiplication.
Now the map 
\[\widetilde{\Phi}_{\K} \colon \widetilde{\mathcal{U}}^+_{\K} \to B_{\Z} \otimes \K, \quad \sum_{n \geq 0} u_n \mapsto \sum_{n \geq 0} \Phi_{\K}(u_n)\]
is a well-defined ring homomorphism which restricts to $\Phi_{\K}$ on $\mathcal{U}_{\K}^+$.

Let $\K$ be a field with at least 4 elements. Then $\widetilde{\mathcal{U}}^+_{\K}$ is a subalgebra of $\widehat{\mathcal{U}}^+_{\K}$ containing $\mfU^+_A(\K)$ as a subgroup of its invertible elements. Indeed, the elements $\exp(\lambda e_i) = \sum_{n \geq 0} e_i^{(n)}\otimes \lambda^n$ belong to $\widetilde{\mathcal{U}}^+_{\K}$, since by assumption $\phi(e_i)$ is nilpotent. Since $A$ is $2$-spherical, together with  \Cref{U+ generated by simple root subgroups} and \Cref{inclusion U+ in Uma+} this implies that 
\[
\mfU^+_A(\K) \cong \left\langle \sum_{n \geq 0} e_i^{(n)} \otimes \lambda^n \mid i \in I, \lambda \in \K \right\rangle \leqslant \left(\widetilde{\mathcal{U}}^+_{\K}\right)^{\times}.
\]
The map $\Phi_{\K} \colon\widetilde{\mathcal{U}}^+_{\K} \to B_{\Z} \otimes \K$ then restricts to the desired group homomorphism on $\mfU^+_A(\K)$.
\end{proof}

From the Serre presentation of $\mfn^+_A(\C)$, as given in \cref{Serre presentation}, it follows that when $A$ is of type $\HB22$, the Lie algebra $\mfn_{A}^+(\C)$ has the presentation
\begin{equation}
\label{presentatie n HB22}
\mfn_A^+(\C) := \left\langle e_a, e_b , e_c \:\Bigg|  \begin{matrix}
\: \ad(e_a)^2(e_b) = \ad(e_b)^2(e_a) = 0 \\
\:  \ad(e_b)^3(e_c) = \ad(e_c)^{2}(e_b) \, = 0 \\
\: \ad(e_a)^3(e_c) = \ad(e_c)^{2}(e_a) = 0
\end{matrix}\right\rangle.
\end{equation}

Let $R := \C\langle X_a, X_b, X_c\rangle$ be the associative unital $\C$-algebra freely generated by the non-commuting variables $X_a, X_b, X_c$. Then $\Mat_{4\times 4}(R)$ is an associative unital $\C$-algebra.
\begin{lemma}
\label{is Lie alg homomorphism}
Let $\mfn^+ := \mfn_A^+(\C)$ be of type $\HB22$. Then the assignment
\begin{align*}
\varphi_{\HB22} \colon  \mfn^+ & \: \to \Mat_{4\times 4}\left(R \right) \\ 
e_a & \: \mapsto \left(\begin{matrix}
0 & 0 & 0 & X_a \\
0 & 0 & X_a & 0 \\
0 & 0 & 0 & 0 \\
0 & 0 & 0 & 0
\end{matrix}\right),\:  e_b \mapsto \left(\begin{matrix}
0 & 0 & 0 & 0 \\
X_b & 0 & 0 & 0 \\
0 & 0 & 0 & -X_b \\
0 & 0 & 0 & 0 
\end{matrix}\right), \: e_c \mapsto \left(\begin{matrix}
0 & 0 & 0 & 0 \\
0 & 0 & 0 & 0 \\
0 & 0 & 0 & 0 \\
0 & X_c & 0 & 0 
\end{matrix}\right),
\end{align*}
extends to a Lie algebra morphism.
\end{lemma}
\begin{proof}
It suffices to verify that the relations from the presentation given by \cref{presentatie n HB22} are satisfied in the image of $\varphi_{\HB22}$. For example,
\begin{align*}
\ad\left(\varphi_{\HB22}(e_c)\right)^2\left(\varphi_{\HB22}(e_a)\right) &= \ad\left(\varphi_{\HB22}(e_c)\right)\left( [X_c E_{4,2}, X_a (E_{1,4} + E_{2,3}) ] \right) \\
& = \left[X_c E_{4,2}, X_c X_a E_{4,3}-X_a X_c E_{1,2} \right] = 0.
\end{align*}
Moreover, 
\begin{align*}
\ad\left(\varphi_{\HB22}(e_a)\right)^2\left(\varphi_{\HB22}(e_c)\right) & = [X_a (E_{1,4} + E_{2,3}), X_a X_c E_{1,2} - X_c X_a E_{4,3}] \\
& = - 2 X_a X_c X_a E_{1,3},
\end{align*}
which implies that
\[
\ad\left(\varphi_{\HB22}(e_a)\right)^3\left(\varphi_{\HB22}(e_c)\right) = [X_a (E_{1,4} + E_{2,3}), - 2 X_a X_c X_a E_{1,3}] = 0.
\]
The remaining relations follow from similar calculations.
\end{proof}

We are now able to prove the first part of \Cref{group hom to SL4K}.
\begin{corollary}
\label{group hom}
Let $\K$ be a field with $\abs{\K} \geq 4$ and let $n \geq 1$. Suppose  $M_a, M_b, M_c \in \Mat_{n \times n}(\K)$. For all $\lambda \in \K$, set
\begin{equation*}
V_a(\lambda)  := \left( \begin{matrix}
1 & 0 & 0 & \lambda M_a \\ 0 & 1 & \lambda M_a & 0 \\ 0 & 0 & 1 & 0 \\ 0 & 0 & 0 & 1
\end{matrix}\right) ,
V_b(\lambda) := \left( \begin{matrix}
1 & 0 & 0 & 0 \\ \lambda M_b & 1 & 0 & 0 \\ 0 & 0 & 1 & -\lambda M_b \\ 0 & 0 & 0 & 1
\end{matrix}\right),  
V_c(\lambda)  := \left( \begin{matrix}
1 & 0 & 0 & 0 \\ 0 & 1 & 0 & 0 \\ 0 & 0 & 1 & 0 \\ 0 & \lambda M_c & 0 & 1
\end{matrix} \right).
\end{equation*}
Then the assignment \[ x_a(\lambda) \mapsto V_a(\lambda), \quad x_b(\lambda) \mapsto V_b(\lambda), \quad x_c(\lambda) \mapsto V_c(\lambda), \]
induces a group homomorphism
\begin{equation*}
\Phi_{M_a,M_b,M_c} \colon  \mfU^+_{\HB22}(\K) \to   \SL_{4n}(\K).
\end{equation*}
\end{corollary}
\begin{proof}
Define $B := \Mat_{4 \times 4}(\C\langle X_a, X_b, X_c \rangle )$. Let $\varphi_{\HB22} \colon \mfn^+ \to B$ be as in \Cref{is Lie alg homomorphism}. Then for $i \in \{a, b , c\}$, 
\[\varphi_{\HB22}(e_i)^2 = 0.\]
Consider the $\Z$-form of $B$ given by $B_{\Z}:=\Mat_{4 \times 4}(\Z\langle X_a, X_b, X_c\rangle )$. Note that for the associated homomorphism of associative algebras $\Phi \colon \mathcal{U}^+_{\C} \to B$,
\[\Phi(\mathcal{U}^+_{\Z}) \subseteq  B_{\Z},\]
since $\Phi(e_i) \in B_{\Z}$ and
\[\Phi(e_i^{(s)}) = \frac{1}{s!} \Phi(e_i)^s = 0 \text{ for all } s >1,\]
for all $i \in \{a,b,c\}$. Thus, by \Cref{prop:alg morphism lifts to group morphism}, for any field $\K$ with at least $4$ elements, there is an associated group homomorphism 
\[\Phi_{\K} \colon \mfU^+_{\HB22}(\K) \to \GL_4(\K\langle X_a, X_b, X_c\rangle),\]
where for each $i \in \{a,b,c\}$, the element $x_i(\lambda) \in \mfU_{i}(\K)$ is mapped to 
\[
\Phi_{\K}(x_i(\lambda)) = (I_4 \otimes 1) + (\phi(e_i) \otimes \lambda). 
\]
For any $M_a, M_b, M_c \in \Mat_{n \times n}(\K)$, and any $\lambda \in \K$, the matrices 
\begin{equation*}
\left( \begin{matrix}
1 & 0 & 0 & \lambda M_a \\ 0 & 1 & \lambda M_a & 0 \\ 0 & 0 & 1 & 0 \\ 0 & 0 & 0 & 1
\end{matrix}\right) ,\quad 
\:\left( \begin{matrix}
1 & 0 & 0 & 0 \\ \lambda M_b & 1 & 0 & 0 \\ 0 & 0 & 1 & -\lambda M_b \\ 0 & 0 & 0 & 1
\end{matrix}\right),  \quad 
\: \left( \begin{matrix}
1 & 0 & 0 & 0 \\ 0 & 1 & 0 & 0 \\ 0 & 0 & 1 & 0 \\ 0 & \lambda M_c & 0 & 1
\end{matrix} \right), 
\end{equation*}
have determinant $1$, and it follows that there is a well-defined group homomorphism
\[
\psi_{M_a,M_b,M_c} \colon \GL_4(\K \langle X_a, X_b, X_c \rangle) \to \SL_{4n}(\K)
\]
defined by mapping $X_i$ to $M_i$ for every $i \in \{a,b,c\}$.  Hence we obtain a group homomorphism
\[\Phi_{M_a,M_b,M_c} \colon \begin{tikzcd}
\mfU^+_{\HB22}(\K) \arrow[r,"\Phi_{\K}"] & \GL_4(\K\langle X_a, X_b, X_c \rangle) \arrow[r,"\psi_{M_a,M_b,M_c}" yshift = 4pt] & \SL_{4n}(\K).
\end{tikzcd} \qedhere\]
\end{proof}

\section{Quotients of \texorpdfstring{$\mfU^+_{\HB22}(\F_q)$}{U+} isomorphic to \texorpdfstring{$\SL_n(\F_q)$}{SLn(Fq)}}
\label{sect: SL4K}
In this section, we construct matrices $M_a, M_b, M_c \in \GL_k(\F_q)$ such that $\Phi_{M_a,M_b,M_c}$ has image $\SL_{4k}(\F_q)$. We achieve this by identifying embedded copies of $\SL_2(\F_{q^k})$ in the image, and applying \Cref{refinement Li}. In the case (iii) of \Cref{refinement Li}, it will be crucial to identify precisely the matrix $\Omega_{B_{f,\phi_x}}$ which defines the copy of $\Sp_{2k}(\F_q)$. To this end, we cite the following result.

\begin{lemma}[{\cite[Proposition 5.1]{Bayer-F}}]
\label{existence self dual basis}
For any Galois extension ${\mathbb{L}}/{\K}$ of odd degree $m$ and characteristic not $2$, there exists a $\K$-basis $\{b_1, \ldots, b_m\}$ of $\mathbb{L}$ such that
\[\Tr_{{\mathbb{L}}/{\K}}(b_ib_j) = \delta_{i,j}.\]
\end{lemma}
Such a basis is called a \emph{self-dual basis}. Moreover, in the case at hand, namely extensions of finite fields of odd degree, it is always possible to choose a self-dual basis in such a way that it is generated by the conjugates of a single element under the Frobenius automorphism. A basis generated by a single element in this way is called a \emph{normal basis}.
\begin{lemma}[{\cite[Theorem 1]{LempelWeinberger}}]
\label{lempelwein}
Let $q$ be a prime power (not necessarily odd). Let $k$ be odd. Then there exists an element $b \in \F_{q^k}$ such that 
\[ \{ b, b^q, b^{q^2}, \ldots, b^{q^{k-1}}\}\] is an $\F_q$-basis of $\F_{q^k}$, and such that 
$\Tr_{{\F_{q^k}}/{\F_q}}(b^{q^{i}+ q^j}) = \delta_{i,j}$.
\end{lemma}
The proof of the above result is constructive, and the authors of \cite{ArnaultF} have implemented the resulting algorithm in Magma, as well as other algorithms constructing bases of the same type.

The reason why self-dual bases are of particular interest for our purposes, is that they allow to identify embeddings of $\SL_2(\F_{q^k})$ in the standard symplectic group $\Sp_{2k}(\F_q, \Omega_k)$.
\begin{lemma}
A self-dual basis $\mathcal{B}$ of a field extension ${\mathbb{L}}/{\K}$ of degree $k$ defines an embedding of $\SL_2(\mathbb{L})$ in the standard symplectic group $\Sp_{2k}(\K,\Omega_k)$. 
\end{lemma}
\begin{proof}
 Let
\[\tau \colon \SL_2(\mathbb{L}) \to \SL_{2k}(\K), \quad \left(\begin{matrix}
w & x \\ y & z
\end{matrix}\right) \mapsto \left(\begin{matrix}
[\mu_w]_{\mathcal{B}} & [\mu_{x}]_{\mathcal{B}} \\
[\mu_y]_{\mathcal{B}} & [\mu_z]_{\mathcal{B}}
\end{matrix}\right).\]
Since $\mathcal{B}$ is an orthonormal basis for the trace form $\Tr := \Tr_{{\mathbb{L}}/{\K}}$, for any $\lambda \in \mathbb{L}$, the matrix $[\mu_{\lambda}]_{\mathcal{B}}$ is given by $(\Tr(\lambda b_ib_j))_{1 \leq i,j \leq k}$. This matrix is symmetric. Then
\begin{align*}
\tau(\left(\begin{matrix}
w & x \\ y & z
\end{matrix}\right))^t & = \left(\begin{matrix}
[\mu_w]_{\mathcal{B}} & [\mu_x]_{\mathcal{B}}\\ [\mu_y]_{\mathcal{B}} & [\mu_z]_{\mathcal{B}} 
\end{matrix}\right)^t = \left(\begin{matrix}
[\mu_w]_{\mathcal{B}}^t & [\mu_y]_{\mathcal{B}}^t \\ [\mu_x]_{\mathcal{B}}^t & [\mu_z]_{\mathcal{B}}^t
\end{matrix}\right) = \left(\begin{matrix}
[\mu_w]_{\mathcal{B}} & [\mu_y]_{\mathcal{B}}\\ [\mu_x]_{\mathcal{B}} & [\mu_z]_{\mathcal{B}} 
\end{matrix}\right) = \tau(\left(\begin{matrix}
w & x \\ y & z
\end{matrix}\right)^t ).
\end{align*}
But for any matrix $ X \in \SL_2(\mathbb{L})$, one has $X^t \Omega_1 X = \Omega_1$, which implies by the above that
\[\tau(X)^t \Omega_k\tau(X) = \tau(X)^t \tau(\Omega_1) \tau(X) = \tau(X^t \Omega_1 X) = \tau(\Omega_1) = \Omega_k.\qedhere\]
\end{proof}
Since all $\Sp_{2k}(\K, \Omega_f)$ are isomorphic, it follows in particular that every $\Sp_{2k}(\K,\Omega_f)$ contains an isomorphic copy of $\SL_{2}(\mathbb{L})$, with $\abs{\mathbb{L} : \K} = k$, whenever a self-dual $\K$-basis of $\mathbb{L}$ exists. In particular this always holds for field extensions $\mathbb{L}/\K = {\F_{q^k}}/{\F_q}$ with $k$ odd, by \Cref{lempelwein}.

Since we will be constructing subgroups of  $\Ima(\Phi_{M_a,M_b,M_c})$ generated by matrices of order $p$, the following result is essential.

\begin{proposition}
\label{Li p-generated}
Let $p > 2$ and $k$ be distinct primes, and let $q = p^r$ for some $r \geq 1$. Let $X \leqslant \SL_{2k}(\F_q)$. Suppose $\Sp_{2k}(\F_q,\Omega_f ) \< X$ for some symplectic form $f \colon \F_q^{2k} \times \F_q^{2k} \to \F_q$. Suppose moreover that $X$ is generated by elements of order a power of $p$. Then $X = \Sp_{2k}(\F_q, \Omega_f)$.
\end{proposition}
\begin{proof}
Suppose $X$ contains $K := \Sp_{2k}(\F_q,\Omega_{f})$ as a normal subgroup and $X = \langle x_1,\ldots, x_n \rangle$ such that each $x_i$ is of order $p^{\ell_i}$, for $\ell_i \geq 1$. Then $K$ (and hence $X$) contains a subgroup $N \leqslant K$ such that $N \cong \SL_2(\F_{q^k})$.

Now $X$ acts by conjugation on $K$, and in particular there is a group homomorphism $X \to \Out(K)$, with kernel $K \cdot  C_X(K)$. The group of outer automorphisms of a finite Chevalley group $\Sp_{2k}(\F_q)$ is generated by a group of \emph{diagonal automorphisms} and a group of \emph{field automorphisms}, see \cite[Section 12.2 \& Theorem 12.5.1]{Carter}. In particular, $\abs{\Out(K)} = 2k$.  Since $2k$ is coprime to $p$ by assumption, and $X$ is generated by elements of order a power of $p$, we conclude that $X = K \cdot C_X(K)$.

If an element of $\SL_{2k}(\F_q)$ centralises $K$, it also centralises $N$, and thus 
\begin{equation}
\label{centra K}
C_X(K) \leqslant C_{\SL_{2k}(\F_q)}(K) \leqslant C_{\SL_{2k}(\F_q)}(N) \leqslant  C_{\GL_{2k}(\F_q)}(N) \leqslant N_{\GL_{2k}(\F_q)}(N).
\end{equation}
By \Cref{normaliser of SL2}, the latter normaliser is isomorphic to $\GL_{2}(\F_{q^k}) \rtimes \Gal({\F_{q^k}}/{\F_q})$, and under this isomorphism, the inclusion of $N$ in $N_{\GL_{2k}(\F_q)}(N)$ coincides with the natural inclusion of $\SL_2(\F_{q^k})$ in $\GL_{2}(\F_{q^k}) \rtimes \Gal({\F_{q^k}}/{\F_q})$. It follows that $C_{\GL_{2k}(\F_q)}(N) \cong \mathcal{Z}(\GL_2(\F_{q^k}))$, since the centraliser of $\SL_2(\F_{q^k})$ in $\GL_2(\F_{q^k}) \rtimes \Gal({\F_{q^k}}/{\F_q})$ is precisely the centre of $\GL_2(\F_{q^k})$.

Hence by \cref{centra K}, $C_{X}(K)$ is isomorphic to a subgroup of the scalar matrices in $\GL_2(\F_{q^k})$, a group of order $q^k-1$.  In particular, since $X = K \cdot C_X(K)$, it now follows that $X/K$ is a group of order coprime to $p$, and hence every generator of $X$ has trivial image in $X/K$. In other words, $X = K = \Sp_{2k}(\F_q, \Omega_{f})$, which concludes the proof.
\end{proof}

Combining \Cref{refinement Li} with \Cref{Li p-generated}, we deduce the following.

\begin{lemma}
\label{new refinement Li}
Let $p > 2$ and $k$ be distinct primes, and let $q = p^r$ for some $r \geq 1$. Let $N$ be the image of the embedding of $\SL_2(\F_{q^k})$ in $\SL_{2k}(\F_q)$ for a fixed choice of $\F_q$-basis $\mathcal{B}$. Let $X \leqslant \SL_{2k}(\F_q)$ be a subgroup containing $N$. Suppose moreover that $X$ is  generated by elements of order a power of $p$. Then precisely one of the following holds:
\begin{enumerate}[(i)]
\item $N \< X$,
\item $X = \SL_{2k}(\F_q)$,
\item There exists an $x \in \F_{q^k}^{\times}$ such that \[X = \Sp_{2k}(\F_q,\Omega_{B_{f,\phi_x}}),\]  where the symplectic form $B_{f,\phi_x}$ is as defined in \cref{definitie B_f} and $f \colon \F_{q^k}^2 \times \F_{q^k}^2 \to \F_{q^k}$ is the standard symplectic form defined by $\Omega_f = \Omega_1$.
\end{enumerate}
\end{lemma}

We record the following useful fact which follows from the work of Dickson.
\begin{lemma}[{\cite[Chapter 2, Theorem 8.4]{Gorenstein}}]
\label{work of Dickson}
Let $S$ be a Singer element in $\GL_k(\F_q)$, with $q = p^r$ for $r \geq 1$ an odd prime power and $k\geq 1$ an integer such that $q^k \neq 9$. Let \[H := \left\langle \left(\begin{matrix}
1 & 0 \\ S & 1
\end{matrix}\right), \left(\begin{matrix}
1 & 1 \\ 0 & 1
\end{matrix}\right) \right\rangle \leqslant \SL_{2k}(\F_q),\]
where each entry is a $k \times k$-block. Then $H$ is isomorphic to $\SL_{2}(\F_{q^k})$. 
\end{lemma}

Then we obtain the following generalisation of \cite[Proposition 7.18]{CCKW}.
\begin{proposition}
\label{CCKW prop 7.18}
Let $p > 2$ and $k$ be distinct primes, and let $q = p^r$ for some $r \geq 1$ be such that $q^k \neq 9$. Suppose $M_1, M_2, M_3 \in \Mat_{k\times k}(\F_q)$. Then the following hold: 
\begin{enumerate}[(i)]
\item If $M_1 M_2$ is a Singer element in $\GL_k(\F_q)$, then the subgroup \[
\left\langle \left(\begin{matrix}
1 & M_1 \\ 0 & 1
\end{matrix}\right), \left(\begin{matrix}
1 & 0 \\ M_2 & 1
\end{matrix}\right) \right\rangle \leqslant \SL_{2k}(\F_q) \]
is isomorphic to $\SL_2(\F_{q^k})$.
\item \label{M1M2M3} If $M_1M_2$ is a Singer element in $\GL_k(\F_q)$ and $M_1 M_2 M_3 \neq M_3 M_2 M_1$, then
\[
\left\langle\left(\begin{matrix}
1 & M_1 \\ 0 & 1
\end{matrix}\right), \left(\begin{matrix}
1 & 0 \\ M_2 & 1
\end{matrix}\right),\left(\begin{matrix}
1 & M_3 \\ 0 & 1
\end{matrix}\right) \right\rangle \leqslant \SL_{2k}(\F_q)
\]
is either isomorphic to $\Sp_{2k}(\F_q)$, or it is the whole of $\SL_{2k}(\F_q)$.
\end{enumerate}
\end{proposition}
\begin{proof}
Let $M_1, M_2 \in \GL_k(\F_q)$ be such that $M_1 M_2$ is a Singer element. Then \Cref{work of Dickson} implies that 
\[G := \left\langle \left(\begin{matrix}
1 & 0 \\
M_1M_2 & 1 
\end{matrix}\right), \left(\begin{matrix}
1 & 1 \\ 0 & 1
\end{matrix}\right) \right\rangle \cong \SL_{2}(\F_{q^k}).\]
Write $d := \diag(1, M_1) \in \GL_{2k}(\F_q)$. Then 
\[d^{-1} G d =  \left\langle \left(\begin{matrix}
1 & M_1 \\
0 & 1
\end{matrix}\right), \left(\begin{matrix}
1 & 0 \\
M_2 & 1
\end{matrix}\right) \right\rangle.\]
The first claim follows. \par
Suppose we have $M_1, M_2, M_3 \in \Mat_{k \times k} (\F_q)$ satisfying the conditions from (\ref{M1M2M3}). If we can show that $\left(\begin{smallmatrix}
1 & M_3 \\ 0 & 1
\end{smallmatrix}\right)$ does not normalise $d^{-1} G d$, the result follows from \Cref{new refinement Li}. Indeed, setting
\[N := d^{-1} G d \leqslant X := \left\langle\left(\begin{matrix}
1 & M_1 \\ 0 & 1
\end{matrix}\right), \left(\begin{matrix}
1 & 0 \\ M_2 & 1
\end{matrix}\right),\left(\begin{matrix}
1 & M_3 \\ 0 & 1
\end{matrix}\right) \right\rangle \leqslant \SL_{2k}(\F_q),\]
if $X$ does not normalise $N$, \Cref{new refinement Li} implies that either $X \cong \Sp_{2k}(\F_q)$ or $X \cong \SL_{2k}(\F_q)$.

We show that  \[Y := d \left(\begin{matrix}
1 & M_3 \\ 0 & 1
\end{matrix}\right) d^{-1} = \left(\begin{matrix}
1 & M_3 M_1^{-1} \\ 0 & 1
\end{matrix}\right)\]
does not normalise $G$, and the conclusion follows. Let $H := N_{\GL_{2k}(\F_q)}(G)$. \Cref{normaliser of SL2} implies that $H \cong \GL_2(\F_{q^k})\rtimes \Gal({\F_{q^k}}/{\F_q})$. Since $\SL_2(\F_{q^k}) = \ker\left(\det \colon \GL_2(\F_{q^k}) \to \F_{q^k}^{\times}\right)$, we obtain that \[\left\lvert\frac{\GL_2(\F_{q^k})}{\SL_2(\F_{q^k})}\right\rvert = q^k-1,\]
and in particular, $\left[H : G \right] = k(q^k-1)$. 
Since $k$ is coprime to $p$ by assumption, it follows that every Sylow $p$-subgroup of $G$ is a Sylow $p$-subgroup of $H$, and since $G$ is normal in $H$, all Sylow $p$-subgroups of $H$ are contained in $G$. 

Let $\alpha \in \F_{q^k}$ be a primitive element. The group $\SL_2(\F_{q^k})$ is isomorphic to $G$, and the isomorphism is given by 
\[
\begin{pmatrix}
w & x \\ y & z
\end{pmatrix} \mapsto \begin{pmatrix}
\omega(w) & \omega(x) \\ \omega(y) & \omega(z)
\end{pmatrix},
\]
where \[\omega \colon \F_{q^k} \to \F_q[M_1 M_2],\quad \alpha \mapsto M_1 M_2. \]
Here $\F_q[M_1M_2]$ denotes the $\F_q$-algebra generated by $M_1M_2$, which is isomorphic to $\F_{q^k}$ via $\omega$ since $M_1 M_2$ is a Singer element. Hence $G$ contains in particular the subgroup 
\[\mathcal{S} := \left\{\left(\begin{matrix}
1 & T \\ 0 & 1
\end{matrix}\right) \mid T\in \F_{q}[M_1 M_2]\right\}. \]
Note that $\mathcal{S}$ is a Sylow $p$-subgroup of $G$, since it is of order $q^k$ and $\abs{G} = \abs{\SL_2(\F_{q^k})} = q^k(q^{2k}-1)$. Hence all Sylow $p$-subgroups of $G$ and thus $H$ are conjugate by an element of $\GL_{2k}(\F_q)$ to $\mathcal{S}$. Now, 
\[
Y \in N_{\GL_{2k}(\F_q)}(G) \iff M_3 M_1^{-1} \in \F_q[M_1 M_2].
\]
Indeed, since $Y$ is of order $p$, it is contained in $H$ if and only if it is contained in some Sylow $p$-subgroup of $H$, if and only if it is contained in some Sylow $p$-subgroup of $G$. Note that $Y$ centralises $\mathcal{S}$, and since $Y$ is of order $p$, this implies that $\langle \mathcal{S}, Y\rangle$ is a $p$-group. Thus $Y$ is contained in a Sylow $p$-subgroup of $G$ if and only if $Y \in \mathcal{S}$, i.e.\ if and only if $M_3M_1^{-1} \in \F_q[M_1M_2]$. 

The $\F_q$-algebra $\F_q[M_1 M_2]$ is commutative. Since $M_1 M_2 M_3 \neq M_3 M_2 M_1$, \[(M_1 M_2) (M_3 M_1^{-1}) \neq M_3 M_2 = (M_3 M_1^{-1}) (M_1 M_2).\]
Hence, $M_3 M_1^{-1} \not\in \F_q[M_1 M_2]$, or in other words, $Y$ does not normalise $G$, as desired.
\end{proof}

We are now able to prove our first key result. 
\begin{proposition}
\label{image contains an SL2k}
Let $p > 2$ and $k > 3$ be distinct primes, and let $q = p^r > 3$ for some $r \geq 1$. Let $M_a, M_b, M_c \in \GL_k(\F_q)$ satisfy the following conditions:
\begin{enumerate}[(i)]
\item $S := (M_aM_b+M_bM_a)M_c = [\mu_{\lambda}]_{\mathcal{B}}$, for some primitive element $\lambda \in \F_{q^k}$, and some self-dual $\F_q$-basis $\mathcal{B} = \{b_1, \ldots, b_k\}$ of $\F_{q^k}$  such that $\Tr(\lambda b_1^2) \neq 0$,
\item
\[M_aM_b M_a^{-1} M_b^{-1} = \left(\begin{matrix}
x & 0 & 0 \\
0 & x^{-1} & 0 \\
0 & 0 & I_{k-2}
\end{matrix}\right),\]
for $x \in \F_q^{\times}$ such that $x \not\in \{\pm 1\}$.
\end{enumerate}
Let $\Phi_{M_a,M_b,M_c} \colon \mfU^+_{\HB22}(\F_q) \to \SL_{4k}(\F_q)$ be the group homomorphism provided by \Cref{group hom}. Then the image of $\Phi_{M_a,M_b,M_c}$ contains an isomorphic copy of $\SL_{2k}(\F_q)$, embedded as matrices of the form
\[\left(\begin{matrix}
1 & 0 & 0 & 0 \\ 0 & \star & 0 & \star \\ 0 & 0 & 1 & 0 \\ 0 & \star & 0 & \star
\end{matrix}\right). \]
\end{proposition}
\begin{proof}
Let $M_a, M_b,M_c \in \GL_k(\F_q)$,  $x \in \F_q^{\times}$ and $\lambda \in \F_{q^k}^{\times}$ be as in the statement. There is an obvious embedding of $\SL_{2k}(\F_q)$ in $\SL_{4k}(\F_q)$, given by 
\begin{equation}
\label{embedding Hc}
\left( \begin{matrix}
w & x \\ y & z
\end{matrix}\right) \mapsto \left(\begin{matrix}
1 & 0 & 0 & 0 \\ 0 & w & 0 & x \\ 0 & 0 & 1 & 0 \\ 0 & y & 0 & z
\end{matrix}\right), 
\end{equation}
so it suffices to show that the image of $\Phi_{M_a,M_b,M_c}$ contains elements of the form 
\[
\left(\begin{matrix}
1 & 0 & 0 & 0 \\ 0 & 1 & 0 & X \\ 0 & 0 & 1 & 0 \\ 0 & 0 & 0 & 1
\end{matrix}\right), \quad  \left(\begin{matrix}
1 & 0 & 0 & 0 \\ 0 & 1 & 0 & 0 \\ 0 & 0 & 1 & 0 \\ 0 & Y & 0 & 1
\end{matrix}\right),
\]
for some $X,Y \in \Mat_{k \times k}(\F_q)$, such that the associated $2k \times 2k$ matrices 
\[
\left(\begin{matrix}
1 & X \\ 0 & 1 
\end{matrix}\right) \text{ and } \left(\begin{matrix}
1 & 0 \\ Y & 1
\end{matrix}\right)
\]
generate $\SL_{2k}(\F_q)$. 

Let $\mathcal{B}:= \{b_1, \ldots, b_k\}$ be a self-dual $\F_q$-basis of $\F_{q^k}$, as provided by \Cref{existence self dual basis}. Let $S := [\mu_{\lambda}]_{\mathcal{B}} \in \GL_k(\F_q)$. Then since $\mathcal{B}$ is an orthonormal basis for the trace form, \[S = (\Tr(\lambda b_i  b_j))_{1 \leq i,j \leq k}.\] In particular, $S$ is a symmetric Singer element of $\GL_k(\F_q)$.  \par 
We will write 
\begin{equation}
\label{Ma Mb Mc}
V_a' := V_a(1), \quad V_b' := V_b(1), \quad V_c' := V_c(1), 
\end{equation}
with 
\[
V_a(\xi)  := \left( \begin{matrix}
1 & 0 & 0 & \xi  M_a \\ 0 & 1 & \xi M_a & 0 \\ 0 & 0 & 1 & 0 \\ 0 & 0 & 0 & 1
\end{matrix}\right) ,
V_b(\xi) := \left( \begin{matrix}
1 & 0 & 0 & 0 \\ \xi M_b & 1 & 0 & 0 \\ 0 & 0 & 1 & -\xi M_b \\ 0 & 0 & 0 & 1
\end{matrix}\right),  
V_c(\xi)  := \left( \begin{matrix}
1 & 0 & 0 & 0 \\ 0 & 1 & 0 & 0 \\ 0 & 0 & 1 & 0 \\ 0 & \xi M_c & 0 & 1
\end{matrix} \right), 
\]
for all $\xi \in \F_q$.

A direct calculation shows that in the image of $\Phi_{M_a,M_b,M_c}$, 
\begin{equation*}
[V_a',V_b'] = \left(\begin{matrix}
1 & 0 & 0 & 0 \\ 0 & 1 & 0 & -M_aM_b -M_bM_a \\ 0 & 0 & 1 & 0 \\ 0 & 0 & 0 & 1
\end{matrix} \right).
\end{equation*}
The condition on \[X := M_a M_b M_a^{-1} M_b^{-1}\] implies that $I_k + X$ is invertible, and in particular $M_aM_b + M_b M_a$ is invertible. Setting $M_c = (M_aM_b +M_b M_a)^{-1} S$, it follows by \Cref{CCKW prop 7.18}~(i) that the image of $\Phi_{M_a,M_b,M_c}$,  under the identification (\ref{embedding Hc}), contains the subgroup
\begin{equation*}
N := \left\langle  \left(\begin{matrix}
1 & 0 \\ M_c & 1
\end{matrix}\right),  \: \left(\begin{matrix}
 1 & -M_aM_b -M_bM_a \\ 0 & 1
\end{matrix}\right) \right\rangle  \cong \SL_2(\F_{q^k}).
\end{equation*}
We now show that one can construct sufficiently many elements in the image of $\Phi_{M_a,M_b,M_c}$ so that together with the generators of $\SL_2(\F_{q^k})$ above, they generate an overgroup of $\SL_{2}(\F_{q^k})$ which is isomorphic to $\SL_{2k}(\F_q)$. \par It is a direct calculation to show that 
\[
[V_c',V_a',V_a',V_b',V_b'] = \left(\begin{matrix}
1 & 0 & 0 & 0 \\ 0 & 1 & 0 & -4 M_b M_a M_c M_a M_b \\ 0 & 0 & 1 & 0 \\ 0 & 0 & 0 & 1
\end{matrix}\right),
\]
and 
\[
[V_c',V_b',V_b', V_a',V_a'] = \left(\begin{matrix}
1 & 0 & 0 & 0 \\ 0 & 1 & 0 & -4 M_a M_b M_c M_b M_a  \\ 0 & 0 & 1 & 0 \\ 0 & 0 & 0 & 1
\end{matrix}\right).
\]
Additionally,
\begin{align*}
Z  & := (I_k+X)^{-1} X S  (I_k+X)^{-1} = (I_k + X^{-1})^{-1} S (I_k + X)^{-1} \\
& = (I_k + M_bM_a M_b^{-1}M_a^{-1})^{-1}(M_aM_b + M_b M_a) M_c  (I_k + M_a M_b M_a^{-1} M_b^{-1})^{-1}\\
& = \left(( M_a M_b+M_b M_a)M_b^{-1} M_a^{-1}\right)^{-1} (M_a M_b + M_b M_a) M_c ((M_b M_a + M_a M_b)M_a^{-1} M_b^{-1})^{-1} \\
& = M_a M_b M_c M_b M_a (M_a M_b + M_b M_a)^{-1}.
\end{align*}
Since $X$ and $S$ are symmetric,  
\begin{align*}
Z^t &  = (I_k+X)^{-1} SX  (I_k+X)^{-1} = (I_k + X)^{-1} S (I_k+X^{-1})^{-1}\\
& = M_b M_a(M_a M_b+ M_b M_a)^{-1} (M_aM_b + M_b M_a)  M_c M_b M_a ( M_a M_b+M_b M_a)^{-1} \\
& = M_b M_a M_c M_b M_a (M_aM_b+M_bM_a)^{-1}.
\end{align*}
Let $\ell \geq 1$ be such that $-4 \ell = 1 \mod p$. Then 
\begin{align*}
\left(\begin{matrix}
1 & M_a M_b M_c M_b M_a  \\ 0 & 1
\end{matrix}\right)  \mapsto [V_c',V_b',V_b',V_a', V_a']^{\ell} &\quad  \text{ and }\\ \left(\begin{matrix}
1 & M_b M_a M_c M_a M_b  \\ 0 & 1
\end{matrix}\right) \mapsto [V_c',V_a',V_a',V_b',V_b']^{\ell} &
\end{align*}
under the identification (\ref{embedding Hc}). Moreover, conjugation by the matrix $\diag(1, M_aM_b +M_b M_a)$ induces the following identifications:
\begin{align*}
\left(\begin{matrix}
1 & 0  \\ M_c & 1
\end{matrix}\right) & \:\mapsto \left(\begin{matrix}
1 & 0 \\
S & 1
\end{matrix}\right), \quad 
\,\left(\begin{matrix}
1 & M_a M_b + M_b M_a \\ 0 & 1
\end{matrix}\right) \, \mapsto \left(\begin{matrix}
1 & 1 \\ 0 & 1
\end{matrix}\right), \\
\left(\begin{matrix}
1 & M_a M_b M_c M_b M_a  \\ 0 & 1
\end{matrix}\right) & \: \mapsto \left(\begin{matrix}
1 & Z \\ 0 & 1
\end{matrix}\right),  \quad 
\left(\begin{matrix}
1 & M_b M_a M_c M_a M_b  \\ 0 & 1
\end{matrix}\right) \mapsto \left(\begin{matrix}
1 & Z^t \\ 0 & 1
\end{matrix}\right).
\end{align*}
In particular, the composition of the inverse of the embedding (\ref{embedding Hc}) and  conjugation by \newline $\diag(1, M_aM_b + M_b M_a)$ induces an isomorphism which sends
\begin{align*}
H := & \left\langle V_c', [V_a',V_b'], [V_c',V_b',V_b',V_a', V_a']^{\ell}[V_c',V_a',V_a',V_b',V_b']^{\ell}\right\rangle  \\
\cong & \Bigg\langle \left(\begin{matrix}
1 & 0  \\ M_c & 1
\end{matrix}\right), \left(\begin{matrix}
1 & M_aM_b +M_bM_a \\ 0 & 1 
\end{matrix}\right), \left(\begin{matrix}
1 &  M_aM_bM_cM_bM_a + M_bM_aM_cM_aM_b \\ 0 & 1 
\end{matrix}\right)\Bigg\rangle
\end{align*}
to 
\begin{equation*}
G := \left\langle  \left(\begin{matrix}
1 & 0 \\ S & 1 
\end{matrix}\right), \left(\begin{matrix}
1 & 1 \\ 0 & 1
\end{matrix}\right), \left(\begin{matrix}
1 & Z + Z^t\\ 0 & 1
\end{matrix}\right)\right\rangle.
\end{equation*}
Note that $Z +Z^t$ is symmetric. We claim that $Z + Z^t \not\in \langle S \rangle = C_{\GL_k(\F_q)}(S)$, where the latter equality holds by \Cref{centraliser of Singer element}. Writing $y = (1+x)^{-1}$ and $z = (1+x^{-1})^{-1}$ (note that $x \not\in\{\pm 1\}$), we obtain
\begin{align*}
Z + Z^t  =  & \: (I+X)^{-1}(XS + SX)(I+X)^{-1} \\
= & \: \left(\begin{matrix}
z & &\\
& y & \\
& & 2^{-1}I_{k-2}
\end{matrix}\right) S \left(\begin{matrix}
y & & \\
& z & \\
& & 2^{-1} I_{k-2}
\end{matrix}\right) 
+ \left(\begin{matrix}
y & & \\
& z & \\
& & 2^{-1} I_{k-2}
\end{matrix}\right) S \left(\begin{matrix}
z & & \\
& y & \\
& & 2^{-1}I_{k-2}
\end{matrix}\right) \\
 = & \: 2^{-1} \left( \begin{array}{@{}c|c@{}}
    \begin{matrix}
      4yz S_{1,1} & 2 \left(y^{2}+z^{2} \right) S_{1,2} \\
      2\left(y^{2}+z^{2}\right) S_{2,1} & 4yz S_{2,2}
   \end{matrix} 
      & \begin{matrix}
      	(y+z)S_{1,3} & \ldots & (y+z)S_{1,k} \\
      	(y+z)S_{2,3} & \ldots & (y+z)S_{2,k}
		\end{matrix}       \\
   \cmidrule[0.2pt]{1-2}
   \begin{matrix}
   (y+z)S_{3,1} & & (y+z)S_{3,2} \\
   \vdots &  &\vdots \\
    (y+z) S_{k,1} & &(y+z) S_{k,2}
   \end{matrix} &  [S_{i,j}]_{3 \leq i,j \leq k} \\
\end{array} \right).
\end{align*}
If $Z + Z^t$ were to commute with $S$, then so would $S - 2(Z+Z^t)$. Moreover, since $\langle S \rangle = C_{\GL_k(\F_q)}(S)$, it would follow that $S - 2(Z+Z^t) \in \F_q[S]$. But by the above, this element is equal to
\begin{align*}
\left( \begin{array}{@{}c|c@{}}
    \begin{matrix}
      (1-4yz) S_{1,1} & (1-2 \left(y^{2}+z^{2} \right) )S_{1,2} \\
     (1- 2\left(y^{2}+z^{2}\right)) S_{2,1} & (1-4yz) S_{2,2}
   \end{matrix} 
      & \begin{matrix}
      	(1-y-z)S_{1,3} & \ldots & (1-y-z)S_{1,k} \\
      	(1-y-z)S_{2,3} & \ldots & (1-y-z)S_{2,k}
		\end{matrix}       \\
   \cmidrule[0.2pt]{1-2}
   \begin{matrix}
   (1-y-z)S_{3,1} & & (1-y-z)S_{3,2} \\
   \vdots &  &\vdots \\
    (1-y-z) S_{k,1} & &(1-y-z) S_{k,2}
   \end{matrix} &  0 \\
\end{array} \right).
\end{align*}
Since $k \geq 5$, this matrix is non-invertible, and since the only non-invertible matrix in $\F_q[S]$ is the zero matrix, it follows that $S - 2(Z+Z^t) = 0$. This implies that $(1-4yz)S_{1,1} = 0$. Since $4yz = 1 \iff 2 +x + x^{-1} = 4$, which holds if and only if $x = 1$, we obtain that $S_{1,1}=0$. This is a contradiction since by assumption, $S_{1,1} = \Tr(\lambda b_1^2) \neq 0$. 

It follows that $Z+Z^t \not\in C_{\GL_k(\F_q)}(S)$, and in particular, by \Cref{CCKW prop 7.18}, $G \cong \Sp_{2k}(\F_q)$ or $G \cong \SL_{2k}(\F_q)$. In the latter case, there is nothing left to prove, since it implies that $H \cong \SL_{2k}(\F_q)$, as desired. So suppose $G \cong \Sp_{2k}(\F_q)$. Then, since every generator of $G$ is contained in the standard symplectic group $\Sp_{2k}(\F_q,\Omega_k)$, it follows that \[G = \Sp_{2k}(\F_q, \Omega_k).\]
Note that $X$ does not commute with $S$. Indeed, if $XS = SX$, then 
\[
\left(\begin{matrix}
xS_{1,1} & x S_{1,2} & x S_{1,3} & \ldots & xS_{1,k} \\
x^{-1} S_{2,1} & x^{-1} S_{2,2} & x^{-1} S_{2,3} & \ldots & x^{-1} S_{2,k} \\
S_{3,1} & S_{3,2} &  S_{3,3} & \ldots & S_{3,k} \\
\vdots & \vdots & \vdots  & \ddots & \vdots \\
S_{k,1} & S_{k,2} & S_{k,3} & \ldots & S_{k,k}
\end{matrix}\right) = \left(\begin{matrix}
xS_{1,1} & x^{-1}S_{1,2} &  S_{1,3} & \ldots & S_{1,k} \\
x S_{2,1} & x^{-1} S_{2,2} & S_{2,3} & \ldots & S_{2,k} \\
x S_{3,1} & x^{-1} S_{3,2} &  S_{3,3} & \ldots & S_{3,k} \\
\vdots & \vdots &  \vdots & \ddots & \vdots \\
x S_{k,1} & x^{-1} S_{k,2} & S_{k,3} & \ldots & S_{k,k}
\end{matrix}\right), 
\]
and hence since $x \neq \pm 1$,
\[
S = \left(\begin{matrix}
S_{1,1} & 0 & 0  \\
 0 & S_{2,2} & 0  \\
 0 &  0 & [S_{i,j}]_{3 \leq i,j \leq k}
\end{matrix}\right),
\]
which is a contradiction to \Cref{Singer not block diag}. In particular,
\[(I + X)^{-1} X S (I+X)^{-1} \neq (I+X)^{-1}SX (I+X)^{-1},\]
and it follows that $Z = (I + X)^{-1} X S (I+X)^{-1}$ is non-symmetric. Hence, 
\[\left(\begin{matrix}
1 & Z \\
 0 & 1
\end{matrix}\right) \not\in \Sp_{2k}(\F_q,\Omega_k).\]
Write $d := \diag(1,M_a M_b + M_b M_a)$. Then 
\[\SL_2(\F_{q^k}) \cong d^{-1} N d \lneq  G = \Sp_{2k}(\F_q, \Omega_k) \lneq \left\langle G, \left(\begin{matrix}
1 & Z \\ 0 & 1
\end{matrix}\right) \right\rangle.\]
The group $G$ does not normalise $d^{-1} N d $, since it is simple modulo its centre. Hence it follows by \Cref{new refinement Li} that
\begin{align*}
\left\langle G, \left(\begin{matrix}
1 & Z \\ 0 & 1
\end{matrix}\right) \right\rangle \cong \SL_{2k}(\F_q).
\end{align*}
But under conjugation by $\diag(1, M_a M_b + M_b M_a)$ and the embedding given by \cref{embedding Hc}, the group above is isomorphic to 
\[ \langle V_c', [V_a',V_b'], [V_c',V_b',V_b',V_a', V_a']^{\ell}, [V_c',V_a',V_a',V_b',V_b']^{\ell}\rangle \leqslant \Ima(\Phi_{M_a,M_b,M_c}).\qedhere\]
\end{proof}

We also record the following result.
\begin{lemma}[{\cite[Theorem 1]{LenstraSchoof}}]
\label{PrimitiveNormalBasis}
For every prime power $q > 1$ and every integer $m$, there exists a normal $\F_q$-basis of $\F_{q^m}$, i.e.\ a basis of the form \[\{b, b^{q}, b^{q^2}, \ldots, b^{q^{m-1}}\},\] such that $b$ (and hence every $b^{q^{i}}$) is primitive.
\end{lemma}
In particular, we obtain the following.
\begin{lemma}
\label{Singer element for SL2 exists}
Let $q > 1$ be a prime power and $k \geq 2$. Suppose $b \in \F_{q^k}$ satisfies $\Tr(b) = 1$. Then there exists a primitive element $\lambda \in \F_{q^k}$ such that $\Tr(\lambda b) \neq 0$.
\end{lemma}
\begin{proof}
Suppose that for all primitive elements $\lambda \in \F_{q^k}$, $\Tr(\lambda b) = 0$. By \Cref{PrimitiveNormalBasis}, there exist primitive elements $\lambda_i \in \F_{q^k}$ and scalars $x_i \in \F_q$ such that 
\[1 = \sum_{i} x_i \lambda_i.\]
But then $1 = \Tr(b) = \sum_{i} x_i \Tr(\lambda_i b) = 0$,
a contradiction.
\end{proof}

\begin{proposition}
There exist matrices $M_a, M_b, M_c \in \GL_{k}(\F_q)$ satisfying the conditions of \Cref{image contains an SL2k}.
\end{proposition}
\begin{proof}
The commutator subgroup $[\GL_k(\F_q), \GL_k(\F_q)]$ of $\GL_k(\F_q)$ is given by $\SL_{k}(\F_q)$. It follows that there exist $M_a, M_b \in \GL_k(\F_q)$ with \[ M_a M_b M_a^{-1} M_b^{-1} = \left(\begin{matrix}
x & & \\
& x^{-1} & \\
 && I_{k-2}
\end{matrix}\right), \]
for any $x \in \F_q^{\times}\setminus\{\pm 1\}$. Such an $x$ always exists, since by assumption $q > 3$. By \Cref{Singer element for SL2 exists}, for any self-dual basis $\mathcal{B} = \{b_1, \ldots, b_k\}$ of $\F_{q^k}$ over $\F_q$, there exists a primitive element $\lambda \in \F_{q^k}$ such that $\Tr_{{\F_{q^k}}/\F_q}(\lambda b_1^2) \neq 0$. The matrix $I_k + M_a M_b M_a^{-1} M_b^{-1}$ is invertible by the condition on $x$. Setting $M_c := (M_aM_b + M_b M_a)^{-1} [\Tr(\lambda b_i b_j)]_{i,j}$, the matrices $M_a$, $ M_b$ and $M_c$ satisfy the desired conditions.
\end{proof}

We now show that the subgroups of $2k \times 2k$ matrices we constructed in the image of $\Phi_{M_a, M_b, M_c}$ give rise to the desired quotients.
\begin{theorem}
\label{onto SL4}
For $q$, $k$, $M_a, M_b$ and $M_c$ as in \Cref{image contains an SL2k}, the associated homomorphism 
\[\Phi_{M_a, M_b,M_c} \colon  \mfU^+_{\HB22}(\F_q) \to \SL_{4k}(\F_q) \]
as provided by \Cref{group hom} is a surjective group homomorphism.
\end{theorem}

\begin{proof}
Let $M_a, M_b$ and $M_c$ be as in \Cref{image contains an SL2k}. The image of $\Phi_{M_a,M_b,M_c}$ contains an isomorphic copy of $\SL_{2k}(\F_q)$, embedded as in \cref{embedding Hc}. In particular, for any matrices $X,Y \in \Mat_{k \times k}(\F_q)$, the image of $\Phi_{M_a,M_b,M_c}$ contains the elements 
\[V_{-\alpha_2}(X) := \left(\begin{matrix}
1 & 0 & 0 & 0 \\ 
0 & 1 & 0 & 0 \\
0 & 0 & 1 & 0 \\
0 & X & 0 & 1
\end{matrix}\right) \quad \text{ and }\quad V_{\alpha_2}(Y) := \left(\begin{matrix}
1 & 0 & 0 & 0 \\ 
0 & 1 & 0 & Y \\
0 & 0 & 1 & 0 \\
0 & 0 & 0 & 1
\end{matrix}\right),\]
and more generally, the subgroup
\[H_c := \left\langle V_{-\alpha_2}(X) , V_{\alpha_2}(Y) \mid X,Y \in \Mat_{k \times k }(\F_q) \right\rangle \cong \SL_{2k}(\F_q).\]
We set, as before, 
\[V_a' := \left(\begin{matrix}
1 & 0 & 0 & M_a \\
0 & 1 & M_a & 0 \\
0 & 0 & 1 & 0 \\
0 & 0 & 0 & 1
\end{matrix}\right) \quad \text{ and } \quad V_b' := \left(\begin{matrix}
1 & 0 & 0 & 0 \\
M_b & 1 & 0 & 0 \\
0 & 0 & 1 & -M_b\\
0 & 0 & 0 & 1
\end{matrix}\right).\]
For any $X \in \Mat_{k \times k}(\F_q)$, one calculates that
\[
[V_b', V_{- \alpha_2}(X), V_b'] =   \left[  I_{4k} -M_b E_{3,2} - XM_b E_{4,1} - M_b X M_b E_{3,1}, V_b' \right] = I_{4k} - 2 M_b X M_b E_{3,1}.
\]
Writing $X := -2^{-1} M_b^{-1} Y M_b^{-1}$ for $Y \in \Mat_{k \times k}(\F_q)$, we deduce that
\[
\left(\begin{matrix}
1 & 0 & 0 & 0 \\
 0 & 1 & 0 & 0 \\
 Y & 0 & 1 & 0 \\
  0 & 0 & 0 & 1
\end{matrix}\right) \in \Ima(\Phi_{M_a,M_b,M_c}), \quad \text{ for all } Y \in \Mat_{k \times k} (\F_q).
\]
Similarly, for $X \in \Mat_{k \times k}(\F_q)$,
\[
[V_a', V_{-\alpha_2}(X), V_a'] = \left[ I_{4k} + M_a X E_{1,2} - XM_a E_{4,3} +M_a X M_a E_{1,3}, V_a \right] = I_{4k} + 2 M_a X M_a E_{1,3}.
\]
Writing $X := 2^{-1}M_a^{-1} Y M_a^{-1}$ with $Y \in \Mat_{k\times k}(\F_q)$, we conclude that there is a second embedded copy of $\SL_{2k}(\F_q)\leqslant \Ima(\Phi_{M_a,M_b,M_c})$, embedded as matrices of the form
\begin{equation}
\label{second SL2}
\left(\begin{matrix}
\star & 0 & \star & 0 \\
0 & 1 & 0 & 0 \\
\star & 0 & \star & 0\\
0 & 0 & 0 & 1 
\end{matrix}\right).
\end{equation}
Together with $H_c$, it follows that $\Ima(\Phi_{M_a,M_b,M_c})$ contains in particular the generators $x_{-\alpha_i}(\lambda) = I_{4k} + \lambda E_{i+1,i}$ and $x_{\alpha_i}(\lambda) = I_{4k} + \lambda E_{i,i+1}$ of the root subgroups of $\SL_{4k}(\F_q)$, for \[i \in  \{1, \ldots, 4k-1\} \setminus \{k, 2k, 3k \}\] and $\lambda \in \F_q$ (see \cref{subsec Kac-Moody grps}). It remains to show that $\Ima(\Phi_{M_a,M_b,M_c})$ contains $x_{\pm\alpha_k}(\lambda)$, $x_{\pm\alpha_{2k}}(\lambda)$ and $x_{\pm\alpha_{3k}}(\lambda)$ for $\lambda \in \F_q$. Since $\Ima(\Phi_{M_a,M_b,M_c})$ contains a copy of $\SL_{2k}(\F_q)$ embedded as in \cref{second SL2}, and a copy of $\SL_{2k}(\F_q)$ embedded as in \cref{embedding Hc} it contains in particular the matrices
\[
\left(\begin{matrix}
X & 0 & 0 & 0 \\
0 & 1 & 0 & 0 \\
0 & 0 & 1 & 0 \\
0 & 0 & 0 & 1
\end{matrix}\right) \text{ and }
\left(\begin{matrix}
1 & 0 & 0 & 0 \\
0 & 1 & 0 & 0 \\
0 & 0 & 1 & 0 \\
0 & 0 & 0 & Y
\end{matrix}\right) \qquad \forall \: X, Y \in \SL_{k}(\F_q).
\]
Conjugating $V_b'$ by the product of the above elements, one obtains
\begin{align*}
& \left(\begin{matrix}
X^{-1} & 0 & 0 & 0 \\
0 & 1 & 0 & 0 \\
0 & 0 & 1 & 0 \\
0 & 0 & 0 & Y^{-1}
\end{matrix}\right) \left(\begin{matrix}
1 & 0 & 0 & 0 \\
M_b & 1 & 0 & 0 \\
0 & 0 & 1 & -M_b \\
0 & 0 & 0 & 1
\end{matrix}\right) \left(\begin{matrix}
X & 0 & 0 & 0 \\
0 & 1 & 0 & 0 \\
0 & 0 & 1 & 0 \\
0 & 0 & 0 & Y
\end{matrix}\right) \\ = & \left(\begin{matrix}
1 & 0 & 0 & 0 \\
M_b X & 1 & 0 & 0 \\
0 & 0 & 1 & -M_b Y \\
0 & 0 & 0 & 1
\end{matrix}\right) =: V_b(X,Y) \in \Ima(\Phi_{M_a,M_b,M_c}), \quad \forall \: X,Y \in \SL_k(\F_q).
\end{align*}
In particular, 
\begin{align*}
& V_b(X,I_k) V_b(X,-I_k) = \left(\begin{matrix}
1 & 0 & 0 & 0 \\
2M_b X & 1 & 0 & 0 \\
0 & 0 & 1 & 0\\
0 & 0 & 0 & 1
\end{matrix}\right) \in \Ima(\Phi_{M_a,M_b,M_c}),  \text{ and } \\ & V_b(I_k,Y)V_b(-I_k,Y) = \left(\begin{matrix}
1 & 0 & 0 & 0 \\ 0 & 1 & 0 & 0 \\ 0 & 0 & 1 & -2M_b Y \\ 0 & 0 & 0 & 1
\end{matrix}\right)\in \Ima(\Phi_{M_a,M_b,M_c}), \quad \forall \: X,Y \in \SL_k(\F_q).
\end{align*}
Let $Z := \diag(1, \det(M_b), 1, \ldots, 1) \in \GL_k(\F_q)$. Define $X := M_b^{-1}(Z + E_{1,k}) \in \SL_k(\F_q)$ and $X' := M_b^{-1} Z \in \SL_k(\F_q)$. Then 
\[
\left(\begin{matrix}
1 & 0 & 0 & 0 \\
2M_b X & 1 & 0 & 0 \\
0 & 0 & 1 & 0\\
0 & 0 & 0 & 1
\end{matrix}\right) \left(\begin{matrix}
1 & 0 & 0 & 0 \\
-2M_b X' & 1 & 0 & 0 \\
0 & 0 & 1 & 0\\
0 & 0 & 0 & 1
\end{matrix}\right) = \left(\begin{matrix}1 & 0 & 0 & 0\\ 2E_{1,k} & 1 & 0 & 0 \\ 0 & 0 & 1 & 0 \\ 0 & 0 & 0 & 1\end{matrix}\right). 
\]
We conclude that $x_{-\alpha_k}(2\lambda) \in \Ima(\Phi_{M_a,M_b,M_c})$, for all $\lambda \in \F_q$. But $q$ is odd and hence $x_{-\alpha_k}(\lambda) \in \Ima(\Phi_{M_a,M_b,M_c})$, for all $\lambda \in \F_q$. 

Setting now $Y := M_b^{-1}(Z+E_{k,1})$ and $Y' := M_b^{-1}Z$ in the products $V_b(I_k, Y) V_b(-I_k,Y)$ and  $V_b(I_k, Y') V_b(-I_k,Y')$, a similar  calculation shows that $x_{\alpha_{3k}}(\lambda) \in \Ima(\Phi_{M_a,M_b,M_c})$ for all $\lambda \in \F_q$. 

Moreover, one calculates that 
\[
[V_a', V_{-\alpha_2}(Y) ] = \left(\begin{matrix}
1 & M_a Y & 0 & 0 \\
0 & 1 & 0 & 0 \\
0 & 0 & 1 & 0 \\
0 & 0 & -Y M_a & 1
\end{matrix}\right) \in \Ima(\Phi_{M_a,M_b,M_c}),\quad  \forall Y \in \Mat_{k \times k}(\F_q).
\]
Now for any $X, Z \in \SL_k(\F_q)$ and writing $Y:= M_a^{-1}$, one has the elements
\begin{align*}
& \left(\begin{matrix}
X & 0 & 0 & 0 \\ 
0 & 1 & 0 & 0 \\
0 & 0 & 1 & 0 \\
0 & 0 & 0 & Z
\end{matrix}\right) \left(\begin{matrix}
1 & M_a Y & 0 & 0 \\
0 & 1 & 0 & 0 \\
0 & 0 & 1 & 0 \\
0 & 0 & -Y M_a & 1
\end{matrix}\right) \left(\begin{matrix}
X^{-1} & 0 & 0 & 0 \\ 
0 & 1 & 0 & 0 \\
0 & 0 & 1 & 0 \\
0 & 0 & 0 & Z^{-1} 
\end{matrix}\right) \\
= & \left(\begin{matrix}
1 & X & 0 & 0 \\
0 & 1 & 0 & 0 \\
0 & 0 & 1 & 0 \\
0 & 0 & -Z  & 1 
\end{matrix}\right) \in \Ima(\Phi_{M_a,M_b,M_c}), \quad \forall X,Z \in \SL_k(\F_q).
\end{align*}
An argument similar to that of the previous paragraph now implies that $x_{\alpha_k}(\lambda)$ and $x_{-\alpha_{3k}}(\lambda) \in \Ima(\Phi_{M_a,M_b,M_c})$, for all $\lambda \in \F_q$. \par 
Finally, by the embedding \cref{second SL2} of $\SL_{2k}(\F_q)$ in the image of $\Phi_{M_a,M_b,M_c}$, and by the previous paragraph, $\Ima(\Phi_{M_a,M_b,M_c})$ contains the elements 
\[
\left(\begin{matrix}
1 & 0 & 0 & 0 \\
0 & 1 & 0 & 0 \\
X & 0 & 1 & 0 \\
0 & 0 & 0 & 1
\end{matrix}\right) \text{ and } \left(\begin{matrix}
1 & Y & 0 & 0 \\
0 & 1 & 0 & 0 \\
0 & 0 & 1 & 0 \\
0 & 0 & 0 & 1
\end{matrix}\right),\quad  \forall \: X,Y \in \Mat_{k \times k}(\F_q).
\]
Taking their commutator, it follows that 
\[\left(\begin{matrix}
1 & 0 & 0 & 0 \\
0 & 1 & 0 & 0 \\
0 & 2 XY & 1 & 0 \\
0 & 0 & 0 & 1
\end{matrix}\right) \in \Ima(\Phi_{M_a,M_b,M_c}), \quad \forall \: X,Y \in \Mat_{k \times k}(\F_q).\] Hence $x_{-\alpha_{2k}}(\lambda) \in \Ima(\Phi_{M_a,M_b,M_c})$, for all $\lambda \in \F_q$. Since the transpose of the above elements is contained in the image as well (again by the embedding of $H_c$ and the previous calculations), we obtain that $x_{\alpha_{2k}}(\lambda) \in \Ima(\Phi_{M_a,M_b,M_c})$, for all $\lambda \in \F_q$. 

We conclude that $\Ima(\Phi_{M_a,M_b,M_c})$ contains all of the root subgroups associated to the simple roots and their opposites. Hence $\Ima(\Phi_{M_a,M_b,M_c}) = \SL_{4k}(\F_q)$ and it follows that $\Phi_{M_a,M_b,M_c}$ is surjective. 
\end{proof}
\section{Quotients of \texorpdfstring{$\mfU^+_{\HB22}(\F_q)$}{U+} isomorphic to \texorpdfstring{$\Sp_{2n}(\F_q)$}{Spn(Fq)}}
\label{sect: SP4k}
We now proceed in a manner similar to that of the previous section to construct quotients of $\mfU^+_{\HB22}(\F_q)$ of symplectic type.
\begin{proposition}
\label{image contains an Sp2k}
Let $p > 2$ and $k > 3$ be distinct primes and let  $q = p^r > 3$ for some $r \geq 1$. Let $M_a, M_b \in \GL_k(\F_q)$ be symmetric matrices, and let $S \in \GL_k(\F_q)$ be a symmetric Singer element. Suppose the following conditions are satisfied:
\begin{enumerate}[(i)]
\item $M_a M_b + M_b M_a$ is a (non-zero) scalar matrix,
\item $M_a M_b S M_b M_a \not\in \langle S \rangle$. \label{ABSBA}
\end{enumerate}
Then, defining $M_c := (M_a M_b + M_b M_a)^{-1} S$, the image of the homomorphism $\Phi_{M_a,M_b,M_c}$ given by \Cref{group hom} is contained in $\Sp_{4k}(\F_q, \Omega_{2k})$, and it contains a copy of $\Sp_{2k}(\F_q, \Omega_k)$, with the embedding given by \cref{embedding Hc}.
\end{proposition}
\begin{proof}
Since $M_a$, $M_b$ and $M_c$ are symmetric matrices by assumption, the associated elements $V_a'$, $V_b'$ and $V_c'$, as given by \cref{Ma Mb Mc}, are contained in $\Sp_{4k}(\F_q, \Omega_{2k})$. \par 
As in the proof of \Cref{image contains an SL2k}, we consider the elements 
\[ [V_a', V_b'] = \left( \begin{matrix}
1 & 0 & 0 & 0 \\
0 & 1 & 0 & -M_aM_b -M_bM_a\\
0 & 0 & 1 & 0 \\
0 & 0 & 0 & 1
\end{matrix}\right)
\]
and 
\[
[V_c', V_b', V_b', V_a', V_a'] = \left(\begin{matrix}
1 & 0 & 0 & 0 \\
0 & 1 & 0 & -4M_a M_b M_c M_b M_a \\
0 & 0 & 1 & 0 \\
0 & 0 & 0 & 1.
\end{matrix}\right).\]
The matrix $M_a M_b M_c M_b M_a$ is symmetric, and thus 
\[\left\langle 
\left(\begin{matrix}
1 & 0 \\
M_c & 1
\end{matrix}\right), \left(\begin{matrix}
1 & -M_aM_b -M_b M_a \\
0 & 1
\end{matrix}\right), \left(\begin{matrix}
1 & -4M_a M_b M_c M_b M_a\\ 0 & 1
\end{matrix}\right)
\right\rangle\leqslant \Sp_{2k}(\F_q, \Omega_k).\]
We embed the above subgroup in $\Ima(\Phi_{M_a,M_b,M_c})$ via \cref{embedding Hc}. In light of \Cref{CCKW prop 7.18}, the statement then follows if 
\[
(M_a M_b + M_b M_a) M_c (M_a M_b M_c M_b M_a) \neq (M_a M_b M_c M_b M_a) M_c (M_a M_b + M_b M_a).
\]
But since $M_a M_b + M_b M_a$ is a scalar matrix by assumption, the above inequality is equivalent to 
\[S M_a M_b S M_b M_a \neq M_a M_b S M_b M_a  S.\]
By \Cref{centraliser of Singer element}, the centraliser of $S$ in $\GL_k(\F_q)$ is given by $\langle S \rangle$. The statement now follows by condition (\ref{ABSBA}).
\end{proof}
We give explicit matrices satisfying the conditions from \Cref{image contains an Sp2k}.
\begin{proposition}
Let $p > 2$ and $k > 3$ be distinct primes. Let $q = p^r > 3$, for some $r \geq 1$. Define
\[
M_a := 
\left( \begin{array}{@{}c|c@{}}
   \begin{matrix}
      -1 & 0 \\
      \: \: 0 & 1
   \end{matrix} 
      & 0     \\
   \cmidrule[0.2pt]{1-2}
   0 & -I_{k-2} \\
\end{array} \right) \text{ and } M_b := \left( \begin{array}{@{}c|c@{}}
   \begin{matrix}
       \: \: \,1 & -1 \\
      -1 & -1
   \end{matrix} 
      & 0     \\
   \cmidrule[0.2pt]{1-2}
   0 & I_{k-2} \\
\end{array} \right).
\]
Then for any symmetric Singer element $S \in \GL_k(\F_q)$, the matrices $M_a$ and $M_b$ satisfy the conditions from \Cref{image contains an Sp2k}.
\end{proposition}
\begin{proof}
The products $M_a M_b$ and $M_b M_a$ are given by 
\[
M_a M_b = \left( \begin{array}{@{}c|c@{}}
   \begin{matrix}
      -1 & 1 \\
      -1 & -1
   \end{matrix} 
      & 0     \\
   \cmidrule[0.2pt]{1-2}
   0 & -I_{k-2} \\
\end{array} \right) , \quad  M_b M_a= \left( \begin{array}{@{}c|c@{}}
   \begin{matrix}
      -1 & -1 \\
      1 & -1
   \end{matrix} 
      & 0     \\
   \cmidrule[0.2pt]{1-2}
   0 & -I_{k-2} \\
\end{array} \right)
\]
and it follows that $M_a M_b + M_b M_a = -2 I_k$. Let $S = [S_{i,j}]_{1 \leq i,j \leq k} \in \GL_k(\F_q)$ be a symmetric Singer element. Then $M_a M_b S M_b M_a $ is given by
\[
\left( \begin{array}{@{}c|c@{}}
   \begin{matrix}
     S_{1,1} - 2 S_{1,2} + S_{2,2} & S_{1,1} - S_{2,2} \\
     S_{1,1} - S_{2,2} &  S_{1,1} + 2S_{1,2} + S_{2,2} 
   \end{matrix} 
      & \begin{matrix}
      S_{1,3} - S_{2,3} & \ldots & S_{1,k} - S_{2,k} \\
      S_{1,3}+ S_{2,3} & \ldots & S_{1,k} + S_{2,k}
		\end{matrix}           \\
   \cmidrule[0.2pt]{1-2}
		\begin{matrix}
			S_{3,1} - S_{3,2} & S_{3,1} + S_{3,2} \\
			\vdots & \vdots \\
			S_{k,1} - S_{k,2} & S_{k,1} + S_{k,2}
		\end{matrix}   
   & [S_{i,j}]_{3 \leq i,j \leq k} \\
\end{array} \right).
\]
Suppose this matrix commutes with $S$. Then in particular by \Cref{centraliser of Singer element} it belongs to $\langle S\rangle $, and hence the matrix $M_a M_b S M_b M_a - S$ belongs to $\F_q[S]$. However, 
\[
M_a M_b S M_b M_a - S = \left( \begin{array}{@{}c|c@{}}
   \begin{matrix}
      - 2 S_{1,2} + S_{2,2} & S_{1,1} - S_{2,2} - S_{1,2} \\
     S_{1,1} - S_{2,2} - S_{1,2} &  S_{1,1} + 2S_{1,2} 
   \end{matrix} 
      & \begin{matrix}
    - S_{2,3} & \ldots &  - S_{2,k} \\
      S_{1,3}& \ldots & S_{1,k} 
		\end{matrix}           \\
   \cmidrule[0.2pt]{1-2}
		\begin{matrix}
			- S_{3,2} & S_{3,1}  \\
			\vdots & \vdots \\
			 - S_{k,2} & S_{k,1} 
		\end{matrix}   
   & 0 \\
\end{array} \right),
\]
which is a singular matrix, since $k \geq 5$. Since the only singular matrix in $\F_q[S]$ is the zero matrix, we conclude that 
\[S = \left( \begin{array}{@{}c|c@{}}
   \begin{matrix}
     S_{1,1}  &  S_{1,2} \\
     S_{1,2}  &   S_{2,2} 
   \end{matrix} 
      & \begin{matrix}
     0 & \ldots &0 \\
     0 & \ldots & 0
		\end{matrix}           \\
   \cmidrule[0.2pt]{1-2}
		\begin{matrix}
			0 & 0 \\
			\vdots & \vdots \\
			0 & 0
		\end{matrix}   
   & [S_{i,j}]_{3 \leq i,j \leq k} \\
\end{array} \right).
\]
But since $S$ is a Singer element, this is a contradiction by \Cref{Singer not block diag}.
\end{proof}
We now show that matrices as in \Cref{image contains an Sp2k} produce quotients of $\mfU^+_{\HB22}(\F_q)$ isomorphic to $\Sp_{4k}(\F_q)$.
\begin{theorem}
\label{quotient of symplectic type}
For $q$, $k$, $M_a$, $M_b$ and $M_c$ as in \Cref{image contains an Sp2k}, the associated homomorphism
\[\Phi_{M_a, M_b,M_c} \colon  \mfU^+_{\HB22}(\F_q) \to \SL_{4k}(\F_q) \]
as defined in \Cref{group hom} has image $\Sp_{4k}(\F_q, \Omega_{2k})$.
\end{theorem}
\begin{proof} 
Let  $M_a$, $M_b$ and $M_c$ be as in \Cref{image contains an Sp2k}. Then $\Ima(\Phi_{M_a,M_b,M_c}) \leqslant \Sp_{4k}(\F_q, \Omega_{2k})$, and the statement implies that
\[
H_c := \langle V_{-\alpha_2}(X) , V_{\alpha_2}(Y) \mid X,Y \in \Mat_{k\times k}(\F_p), X^t = X, Y^t = Y \rangle \cong \Sp_{2k}(\F_q,\Omega_k)
\]
where the isomorphism is given by the inclusion as in \cref{embedding Hc}, and where
\[V_{-\alpha_2}(X) := \left(\begin{matrix}
1 & 0 & 0 & 0 \\ 
0 & 1 & 0 & 0 \\
0 & 0 & 1 & 0 \\
0 & X & 0 & 1
\end{matrix}\right), \qquad V_{\alpha_2}(Y) := \left(\begin{matrix}
1 & 0 & 0 & 0 \\ 
0 & 1 & 0 & Y \\
0 & 0 & 1 & 0 \\
0 & 0 & 0 & 1
\end{matrix}\right).\]
As in the proof of \Cref{onto SL4}, for any symmetric matrix $X \in \Mat_{k \times k}(\F_q)$, one has
\[
[V_b', V_{- \alpha_2}(X), V_b'] =  I_{4k} - 2 M_b X M_b E_{3,1} \in \Ima(\Phi_{M_a,M_b,M_c}).
\]
Similarly, for a symmetric matrix $X \in \Mat_{k \times k}(\F_q)$,
\[
[V_a', V_{-\alpha_2}(X), V_a'] = I_{4k} + 2 M_a X M_a E_{1,3} \in \Ima(\Phi_{M_a,M_b,M_c}).
\]
Setting respectively  $X := 2^{-1}M_a^{-1} Y M_a^{-1}$ and $X := -2^{-1} M_b^{-1} Y M_b^{-1}$ in the above equalities, for $Y \in \Mat_{k \times k}(\F_q)$ symmetric, we obtain that $\Ima(\Phi_{M_a,M_b,M_c})$ contains the elements
\[
\left(\begin{matrix}
1 & 0 & 0 & 0 \\
0 & 1 & 0 & 0 \\
Y & 0  & 1 & 0 \\
0 & 0 & 0 & 1
\end{matrix}\right) \quad \text{ and } \quad \left(\begin{matrix}
1 & 0 & Y & 0 \\
0 & 1 & 0 & 0 \\
0 & 0 & 1 & 0 \\
 0 & 0 & 0 & 1
\end{matrix}\right) \quad \text{ for all } Y \in \Mat_{k \times k}(\F_q) \text{ symmetric}.
\]
We find that  $\Ima(\Phi_{M_a,M_b,M_c})$ contains a second copy of $\Sp_{2k}(\F_q,\Omega_k)$, embedded as matrices of the form (\ref{second SL2}).

In particular, $\Ima(\Phi_{M_a,M_b,M_c})$ contains the generators (see \cref{subsec Kac-Moody grps})
\[x_{\alpha_i}(\lambda) = \left(\begin{matrix}
I_{2k} + \lambda E_{i,i+1} & 0 \\
0 & I_{2k} -\lambda E_{i+1,i}
\end{matrix} \right), \quad x_{-\alpha_i}(\lambda) = \left(\begin{matrix}
I_{2k} + \lambda E_{i+1,i} & 0 \\
0 & I_{2k} - \lambda E_{i,i+1}
\end{matrix} \right),
\]
for $\lambda \in \F_q$, and $i \in \{1, \ldots k-1\} \cup \{k+1, \ldots, 2k-1\}$, where each block is a $2k \times 2k$-matrix. Moreover, $\Ima(\Phi_{M_a,M_b,M_c})$ contains
\[
x_{\alpha_{2k}}(\lambda) = \left(\begin{matrix}
I_{2k} & \lambda E_{2k,2k} \\
0 & I_{2k} \\
\end{matrix}\right), \quad 
x_{-\alpha_{2k}}(\lambda) = 
\left(\begin{matrix}
I_{2k} & 0 \\
\lambda E_{2k,2k} & I_{2k}
\end{matrix}\right), \quad \forall \lambda \in \F_q.
\] 
It remains to show that $\Ima(\Phi_{M_a,M_b,M_c})$ contains the elements
\[
x_{\alpha_k}(\lambda) = \left(\begin{matrix}
I_k & \lambda E_{k,1} & 0 & 0 \\
0 & I_k & 0 & 0 \\
0 & 0 & I_k & 0 \\
0 & 0 & -\lambda E_{1,k} & I_k
\end{matrix}\right), \quad x_{-\alpha_k}(\lambda) = \left(\begin{matrix}
I_k & 0 & 0 & 0 \\
\lambda E_{1,k} & I_k & 0 & 0 \\
0 & 0 & I_k & -\lambda E_{k,1}\\
0 & 0 &  0 & I_k
\end{matrix}\right), \: \forall \lambda \in \F_q.
\]
Since $\Ima(\Phi_{M_a,M_b,M_c})$ contains a copy of $\Sp_{2k}(\F_q,\Omega_k)$ embedded as matrices of the form (\ref{second SL2}), it contains in particular the elements $X E_{1,1} + (X^{-1})^t E_{3,3}$, for all $X \in \GL_k(\F_q)$. Letting $X:= M_b^{-1} Y$ for some $Y \in \GL_k(\F_q)$, we obtain that $\Ima(\Phi_{M_a,M_b,M_c})$ contains
\[
\left(\begin{matrix}
X^{-1} & 0 & 0 & 0 \\
 0 & 1 & 0 & 0 \\
 0 & 0 & X^t & 0 \\
 0 & 0 & 0 & 1
\end{matrix}\right) \left(\begin{matrix}
1 & 0 & 0 & 0 \\
M_b & 1 & 0 & 0 \\
0 & 0 & 1 & -M_b \\
0 & 0 & 0 & 1
\end{matrix}\right)\left(\begin{matrix}
X & 0 & 0 & 0 \\
0 & 1 & 0 & 0 \\
0 & 0 & (X^{-1})^t & 0 \\
0 & 0 & 0 & 1
\end{matrix}\right) = \left(\begin{matrix}
1 & 0 & 0 & 0 \\
Y & 1 & 0 & 0 \\
0 & 0 & 1 & -Y^t \\
0 & 0 & 0 & 1
\end{matrix}\right),
\]
for all $Y \in \GL_k(\F_q)$. Letting $Y := I_k + \lambda E_{1,k}$ for $\lambda \in \F_q$, and $Y' := -I_k$, we obtain that 
\[
\left(\begin{matrix}
1 & 0 & 0 & 0 \\
Y & 1 & 0 & 0 \\
0 & 0 & 1 & -Y^t \\
0 & 0 & 0 & 1
\end{matrix}\right) \left(\begin{matrix}
1 & 0 & 0 & 0 \\
Y' & 1 & 0 & 0 \\
0 & 0 & 1 & -(Y')^t \\
0 & 0 & 0 & 1
\end{matrix}\right) = \left(\begin{matrix}
1 & 0 & 0 & 0 \\
\lambda E_{1,k} & 1 & 0 & 0 \\
0 & 0 & 1 & -\lambda E_{k,1}\\
0 & 0 & 0 & 1
\end{matrix}\right) \in \Ima(\Phi_{M_a,M_b,M_c}).
\]

Similarly to the proof of \Cref{onto SL4}, $\Ima(\Phi_{M_a,M_b,M_c})$ contains the elements 
\[
[V_a', V_{-\alpha_2}(Y)] = \left(\begin{matrix}
1 & M_a Y & 0 & 0 \\
0 & 1 & 0 & 0 \\
0 & 0 & 1 & 0 \\
0 & 0 & -Y M_a & 1
\end{matrix}\right), \quad \forall Y \in \Mat_{k \times k}(\F_q) \text{ symmetric}.
\]
Additionally, for any $X \in \GL_{k}(\F_q)$, the image of $\Phi_{M_a,M_b,M_c}$ contains
\[
\left(\begin{matrix}
X^{-1} & 0 & 0 & 0 \\
 0 & 1 & 0 & 0 \\
 0 & 0 & X^t & 0 \\
 0 & 0 & 0 & 1
\end{matrix}\right) [V_a', V_{-\alpha_2}(Y)]\left(\begin{matrix}
X & 0 & 0 & 0 \\
0 & 1 & 0 & 0 \\
0 & 0 & (X^{-1})^t & 0 \\
0 & 0 & 0 & 1
\end{matrix}\right) = \left(\begin{matrix}
1 & X M_a Y & 0 & 0 \\
0 & 1 & 0 & 0 \\
0 & 0 & 1 & 0 \\
0 & 0 & -Y M_a X^t & 1
\end{matrix}\right).
\]
Setting $Y := M_a^{-1}$ and $X := I_k + \lambda E_{k,1}$ with $\lambda \in \F_q$, and $X' := -I_k$, we obtain that $\Ima(\Phi_{M_a,M_b,M_c})$ contains
\[
\left(\begin{matrix}
1 & X & 0 & 0 \\
0 & 1 & 0 & 0 \\
0 & 0 & 1 & 0 \\
0 & 0 & -X^{t} & 1
\end{matrix}\right)\left(\begin{matrix}
1 & X' & 0 & 0 \\
0 & 1 & 0 & 0 \\
0 & 0 & 1 & 0 \\
0 & 0 & -(X')^{t} & 1
\end{matrix}\right) = \left(\begin{matrix}
1 & \lambda E_{k,1} & 0 & 0 \\
0 & 1 & 0 & 0 \\
0 & 0 & 1 & 0 \\
0 & 0 & -\lambda E_{1,k} & 1
\end{matrix}\right).
\]
Hence $\Ima(\Phi_{M_a',M_b',M_c'})$ contains all the generators of $\Sp_{4k}(\F_q, \Omega_{2k})$ associated to the simple roots and their opposites as given in \cref{subsec Kac-Moody grps}, and we conclude that \[\Ima(\Phi_{M_a,M_b,M_c}) = \Sp_{4k}(\F_q, \Omega_{2k}).\qedhere\]  
\end{proof}

\section{Spectral high-dimensional expanders from Lie type groups}
\label{sect:HDX}
In this section, we prove that our quotients lead to the construction of high-dimensional expanders. First, we introduce some terminology surrounding Kac-Moody-Steinberg groups.
\subsection{Kac-Moody-Steinberg groups}
\label{sect KMS groups}
Throughout this section, let $A = (A_{ij})_{i,j \in I}$ be a $2$-spherical GCM with associated root system $\Delta$. Let $\K$ be a field. We denote by $U^+$ the group $\mfU^+_A(\K)$, and $U_i := \mfU_{\alpha_i}(\K)$ for all $i \in I$. Additionally, we set for all $i \neq j \in I$, \[U_{i,j} := \left\langle \mfU_{\gamma}(\K) \mid \gamma \in \Delta \cap ({\N}{\alpha_i} + {\N}{\alpha_j})\right\rangle \leqslant U^+.\] 
Since $A$ is of $2$-spherical type, each $U_{i,j}$ is the unipotent radical of the standard Borel subgroup of a Chevalley group of type $A_1 \times A_1$, $A_2$, $B_2$ or $G_2$, depending on whether $A_{ij}A_{ji} = 0,1,2$ or $3$.
\begin{definition}
The \emph{Kac-Moody-Steinberg group} (or \emph{KMS-group}) $\KMS{A}(\K)$ of type $A$ over $\K$ is the direct limit of the system $\{U_i, U_{i,j} \mid i\neq j \in I\}$ with respect to the canonical inclusions $U_i \into U_{i,j}$ for all $i \neq j \in I$.
\end{definition}
For each $J \subseteq I$  spherical, we call the subgroup $U_J := \langle U_i \mid i \in J \rangle \leqslant \KMS{A}(\K)$ a \emph{local group}.

\begin{example}
Let $A$ be the GCM of the root system $\Delta$ of type $\HB22$, with simple roots $\{a,b,c\}$ (see \cref{HB22}). In other words, 
\begin{equation}
\label{GCM HB22}
A = \left(\begin{matrix}
2 & -1 & -2\\
-1 &  2 & -2 \\
-1 & -1 & 2
\end{matrix}\right).
\end{equation}
For $\alpha \in \Delta_+$, let $x_{\alpha}: (\K,+) \to \mfU_A^+(\K)$ be as in \cref{subsec Kac-Moody grps}. Then 
\[\KMS{\HB22}(\K) := \varinjlim_{U_i, U_j \hookrightarrow U_{i,j}}\{U_i,U_{i,j} \mid 1 \leq i < j \leq 3\},\]
where \[U_a := U_1 = \langle x_a(\lambda) \mid \lambda \in \K \rangle, \quad U_b := U_2 = \langle x_b(\lambda) \mid \lambda \in \K \rangle, \text{ and } U_c := U_3 = \langle x_c(\lambda) \mid \lambda \in \K \rangle\]
and the groups $U_{i,j}$ are as follows (cf. \Cref{unique expression U+}):
\[U_{a,b} := U_{1,2} = \left\{x_{a}(\lambda_1) x_{b}(\lambda_2)x_{a+b}(\lambda_3) \mid \lambda_1,\lambda_2,\lambda_3 \in \K\right\} \cong \mfU^+_{A_2}(\K) \leqslant \SL_3(\K), \vspace{-10pt}\]
\begin{align*}
U_{b,c} & := U_{2,3} = \{x_{b}(\lambda_1) x_{c}(\lambda_2)x_{b+c}(\lambda_3) x_{2b+c}(\lambda_4) \hspace{-8pt} & \mid & \:\lambda_1,\lambda_2,\lambda_3, \lambda_4 \in \K\} \cong \mfU^+_{B_2}(\K) \leqslant \Sp_4(\K), \\
U_{a,c} & := U_{1,3} = \{x_{a}(\lambda_1) x_{c}(\lambda_2)x_{a+c}(\lambda_3) x_{2a+c}(\lambda_4) \hspace{-8pt} & \mid &\: \lambda_1,\lambda_2,\lambda_3, \lambda_4 \in \K\} \cong \mfU^+_{B_2}(\K) \leqslant \Sp_4(\K).
\end{align*}

\end{example}
By \Cref{U+ generated by simple root subgroups}, $\mfU^+_A(\K)$ is generated by the subgroups $\mfU_{\alpha_i}(\K)$ for $i \in I$ when $\abs{\K} \geq 4$.
In particular, when $\abs{\K} \geq 4$, the canonical morphism
\begin{equation}
\Theta_{\K,A} \colon \KMS{A}(\K) \to \mfU_A^+(\K)
\end{equation}
where each $U_i$ with $i \in I$ is  mapped to $\mfU_{\alpha_i}(\K) \leqslant \mfU_A^+(\K)$, is surjective.

\subsection{High-dimensional expanders from \texorpdfstring{$\KMS{\HB22}(\F_q)$}{the KMS-group}}
In this section, we prove the second part of \Cref{group hom to SL4K}. In particular, our constructions satisfy the hypotheses of \cite[Theorem 4.3]{IngaPaper}, and thus yield infinite families of bounded degree spectral high-dimensional expanders. 
\begin{theorem}
\label{local injectivity property}
Let $\Phi_{M_a, M_b,M_c}$ be as in \Cref{group hom}, with $M_a, M_b$ and $M_c$ invertible. Then $\Phi_{M_a,M_b,M_c}$ is a group homomorphism which is injective on the rank $2$ subgroups of $\mfU^+_{\HB22}(\K)$. Write \[ \Psi := \Theta_{\K, \HB22} \circ \Phi_{M_a,M_b,M_c}.\] Then $\Psi$ satisfies the \emph{intersection property} \[\Psi(U_J \cap U_K) = \Psi(U_J) \cap \Psi(U_K)\]
for all spherical subsets $J,K \subseteq I$.
\end{theorem}
\begin{proof}
Let $M_a,M_b, M_c \in \GL_n(\K)$. Since $\Theta_{\K,\HB22}$ isomorphically maps each $U_{J}$ with $J \subseteq \{a,b,c\}$ and $\abs{J} = 2$ onto the rank $2$ subgroup $\mfU_{A_J}^+(\K)$, it suffices to show that $\Psi$ is injective on the rank $2$ local groups of $\KMS{\HB22}(\K)$.

The images of the rank $2$ local groups $U_{a,b}$, $U_{b,c}$ and $U_{a,c}$ under $\Psi$ are respectively given by $\langle V_a(\lambda), V_b(\lambda)\mid \lambda \in \K \rangle$, $\langle V_b(\lambda), V_c(\lambda) \mid \lambda \in \K \rangle $ and $\langle V_a(\lambda),V_c(\lambda) \mid \lambda \in \K\rangle$. By definition,
\[
U_{a,b} \cong \mfU^+_{A_2}(\K), 
\]
and each element $u \in U_{a,b}$  can be uniquely written as a product
\[u = x_{a}(\lambda_1) x_{b}(\lambda_2) x_{a+b} (\lambda_3), \quad \lambda_1, \lambda_2, \lambda_3 \in \K.\]
Using the relation \[ [x_{a}(\chi),x_{b}(\mu)]= x_{a+b}(\chi \mu), \quad \text{for all } \chi, \mu \in \K,\] (see \cref{commutator relations chevalley}) to calculate the image of $x_{a+b}(\lambda_3)$, it follows that the image of $U_{a,b}$ is given by 
\begin{align}
\Psi(U_{a,b})&  =\left\{\Psi\left(x_{a}(\lambda_1) x_{b}(\lambda_2) x_{a+b}(\lambda_3) \right)\mid \lambda_i \in \K\right\} 
\notag \\
 & = \left\{\left(\begin{matrix}
1 & 0 & 0 & \lambda_1 M_a \\
\lambda_2 M_b & 1 & \lambda_1 M_a & \lambda_3 M_a M_b + \lambda_1 \lambda_2 \lambda_3 M_b M_a \\
0 & 0 & 1 & -\lambda_2 M_b \\
0 & 0 & 0 & 1
\end{matrix}\right) \mid \lambda_i \in \K\right\}. \label{local grp 1}
\end{align}
Since $M_a$ and $M_b$ are invertible matrices and thus in particular $M_a$, $M_b$ and $M_a M_b$ are different from zero, it follows that 
\[\Psi(u) = 1 \iff \lambda_1 = \lambda_2 = \lambda_3 = 0 \iff u = 1.\]
Hence $\Psi$ is injective on $U_{a,b}$.

Similarly, $U_{a,c}$ is isomorphic to $\mfU^+_{B_2}(\K)$. Using the commutation relations given by \cref{commutator relations chevalley}, one has that
\[[ x_{c}(\lambda_1),x_{a}(\lambda_2)] =x_{a+c}(\lambda_1\lambda_2) x_{2a+c}(\lambda_1\lambda_2^2),\]
and 
\[ [x_a(\lambda_1),x_{a+c}(\lambda_2)] = x_{2a+c}( \lambda_1 \lambda_2),\]
for all $\lambda_1, \lambda_2 \in \K$. It is then a direct computation to conclude that
\begin{align}
\Psi &(U_{a,c}) = \left\{ \Psi\left(x_{a}(\lambda_1)x_{c}(\lambda_2)x_{a+c}(\lambda_3) x_{2a+c}(\lambda_4)\right)\mid \lambda_i \in \K \right\}
\notag \\
 = &\left\{ \left(\begin{matrix}
1 & (\lambda_1\lambda_2 + \lambda_3) M_aM_c & (\lambda_4-\lambda_1 \lambda_3) M_a M_c M_a & \lambda_1 M_a \\
0 & 1 & \lambda_1 M_a & 0 \\
0 & 0 & 1 & 0 \\
0 & \lambda_2 M_c & -\lambda_3 M_c M_a & 1
\end{matrix}\right) \mid \lambda_i \in \K \right\}. \label{local grp 2}
\end{align}
Since none of the entries $M_a$, $M_c$, $M_a M_c$, $M_a M_c M_a$ are zero, one has again that $\Phi(u) = 1 $ if and only if $\lambda_i = 0$ for all $i \in \{1,\ldots,4\}$. Hence $\Psi$ is injective on $U_{a,c}$.

A similar calculation now shows that $\Psi(U_{b,c})$ is of the form
\begin{align}
\Psi&(U_{b,c}) = \left\{\Psi\left(x_{b}(\lambda_1) x_{c}(\lambda_2) x_{b+c}(\lambda_3) x_{2b+c}(\lambda_4)\right)\mid \lambda_i \in \K \right\} \notag \\
= & \left\{
\left(\begin{matrix}
1 & 0 & 0 & 0 \\
\lambda_1 M_b & 1 & 0 & 0 \\
(\lambda_4 - \lambda_1 \lambda_3)M_b M_c M_b & (-\lambda_1 \lambda_2 + \lambda_3) M_b M_c & 1 & -\lambda_1 M_b \\
\lambda_3 M_c M_b & \lambda_2 M_c & 0 & 1
\end{matrix}\right)\mid \lambda_i \in \K \right\}.
\label{local grp 3}
\end{align}
Again thanks to the fact that $M_b$ and $M_c$ are invertible, $\Psi$ is injective on $U_{b,c}$, and we conclude that $\Psi$ is injective on the local groups.

The intersection property trivially follows from eqs. (\ref{local grp 1}), (\ref{local grp 2}) and (\ref{local grp 3}).
\end{proof}

The \emph{coset complex} $\CC(G,\mathcal{H})$ (see \cite[Definition 2.4]{IngaPaper}) associated to a group $G$ and a collection of subgroups $\mathcal{H}:=\{H_1, \ldots, H_n\}$, is the simplicial complex with vertex set $\bigsqcup_{i=1}^n G/H_i$, and with $(k-1)$-simplices the subsets
\[
\{g_{i_1} H_{i_1}, \ldots, g_{i_k}H_{i_k}\}, \text{ such that } \bigcap_{j = 1}^k g_{i_j}H_j \neq \emptyset.
\]

\begin{corollary}
\label{HDX1}
Let $p > 2$ and $k > 3$ be distinct primes. Let $q = p^r > 3$ with $r \geq 1$. Suppose $M_a,M_b,M_c \in \GL_k(\F_q)$ are as in \Cref{image contains an SL2k}. Then for the associated matrices $V_a',V_b',V_c'$ as defined in \cref{Ma Mb Mc}, defining 
\[H_T := \langle  \{V_a',V_b',V_c'\} \setminus \{T\} \rangle \leqslant \SL_{4k}(\F_q), \quad T \in \{V_a',V_b',V_c'\},\]
the coset complex 
\[\mathcal{X}_{q,k} := \CC \left(\SL_{4k}(\F_q), \{H_T\}_{T \in \{V_a',V_b',V_c'\}}\right)\]
is a $\frac{\sqrt{2q}+2}{q-2}$-spectral high-dimensional expander of dimension $2$. 
\end{corollary}
\begin{corollary}
\label{HDX2}
Let $p > 2$ and $k > 3$ be distinct primes. Let $q = p^r > 3$ with $r \geq 1$. Suppose $M_a,M_b,M_c \in \GL_k(\F_q)$ are as in \Cref{image contains an Sp2k}.  Then for the associated matrices $V_a',V_b',V_c'$ as defined in \cref{Ma Mb Mc}, the coset complex
\[\mathcal{X}'_{q,k} := \CC\left(\Sp_{4k}(\F_q), \{H_T\}_{T \in \{V_a',V_b',V_c'\}}\right)\]
is a $\frac{\sqrt{2q}+2}{q-2}$-spectral high-dimensional expander of dimension $2$.
\end{corollary}
In particular, setting $P_p := \{k \text{ prime} \mid p \neq k \text{ and } k > 3\}$, we obtain new infinite families $(\mathcal{X}_{q,k})_{k \in P_p}$ and $(\mathcal{X}_{q,k}')_{k \in P_p}$ of bounded degree, spectral high-dimensional expanders, for all $p > 2$ and all $q = p^r > 3$  (with $r \geq 1$).

\begin{proof}[Proof of \Cref{HDX1,HDX2}]
Let $\Phi_{M_a,M_b, M_c}$ be as in \Cref{onto SL4} or \Cref{quotient of symplectic type}. In particular $\K = \F_q$ has size at least $4$. Since $\HB22$ is $2$-spherical, it follows that the canonical map \[\Theta_{\K,\HB22} \colon \KMS{\HB22}(\K) \to \mfU_{\HB22}^+(\K)\] is surjective. Then the quotients of $\mfU^+_{\HB22}(\K)$ from \Cref{onto SL4,quotient of symplectic type} are quotients of $\KMS{\HB22}(\K)$ as well. It follows from \Cref{local injectivity property} that $\Psi := \Theta_{\K,\HB22} \circ \Phi_{M_a,M_b,M_c}$ is injective on the local groups and satisfies the intersection property. Then \cite[Theorem 4.3]{IngaPaper} implies that the coset complexes arising from $\Psi$ yield infinite families of bounded degree, spectral high-dimensional expanders.
\end{proof}

\bibliographystyle{alpha}
\bibliography{references}

\end{document}